\documentclass[12pt,a4paper]{amsart}

\usepackage[margin=2.5cm,marginpar=2cm]{geometry}

\usepackage{hyperref}

\usepackage{stmaryrd}
\usepackage{multirow}
\usepackage[normalem]{ulem}

\usepackage{mathrsfs}

\usepackage[abbrev]{amsrefs}

\usepackage{bbold}
\usepackage{amssymb}
\usepackage{amsmath}
\usepackage{amsthm}
\usepackage{graphicx}
\usepackage{braket}
\usepackage{paralist}
\usepackage[all,cmtip]{xy}
\usepackage{rotating}
\usepackage{arydshln}

\usepackage{color}
\usepackage[dvipdfx,cmyk,table]{xcolor}
\usepackage{mathrsfs}
\usepackage{mathbbol}
\usepackage{cleveref}

\usepackage{mathtools}

\DeclareMathAlphabet{\mathpzc}{OT1}{pzc}{m}{it}

\definecolor{refkey}{gray}{.3}
\definecolor{labelkey}{gray}{.3}

\newcommand{\tG}{\widetilde{G}}
\newcommand{\talpha}{\tilde{\alpha}}

\newcommand{\AAA}{{\mathscr{A}}}

\newcommand{\rA}{\mathrm{A}}

\newcommand{\val}{{\nu}}

\newcommand{\Jump}{\mathrm{Jump}}

\newcommand{\rO}{\mathrm{O}}

\newcommand{\pr}{\mathrm{pr}}

\newcommand{\bfG}{\mathbf{G}}

\newcommand{\bfH}{\mathbf{H}}

\newcommand{\bfS}{\mathbf{S}}
\newcommand{\bfT}{\mathbf{T}}

\newcommand{\sfg}{\mathsf{g}}
\newcommand{\csfg}{{\check{\sfg}}}

\newcommand{\sfL}{\mathsf{L}}

\newcommand{\sfA}{\mathsf{S}}

\newcommand{\sfG}{\mathsf{G}}

\newcommand{\sfJ}{\mathsf{J}}
\newcommand{\sfM}{\mathsf{M}}

\newcommand{\sfP}{\mathsf{P}}
\newcommand{\sfQ}{\mathsf{Q}}
\newcommand{\sfS}{\mathsf{S}}

\newcommand{\sfV}{\mathsf{V}}

\newcommand{\sfW}{\mathsf{W}}
\newcommand{\sfX}{\mathsf{X}}
\newcommand{\sfY}{\mathsf{Y}}

\newcommand\iso{\xrightarrow{
   \,\smash{\raisebox{-0.65ex}{\ensuremath{\scriptstyle\sim}}}\,}}

\newcommand\riso{\xleftarrow{
   \,\smash{\raisebox{-0.65ex}{\ensuremath{\scriptstyle\sim}}}\,}}

\def\tt{\tilde{t}}

\def\rid{{\mathrm{id}}}

\def\ta{\tilde{a}}

\def\barff{{\overline{f}}}

\def\barww{{\overline{w}}}
\def\barkk{{\overline{k}}}

\def\Stab{{\rm Stab}}
\def\ad{{\rm ad}}
\def\Ad{{\rm Ad}}

\def\id{{\rm id}}

\makeatletter
\def\inn#1#2{\left\langle 
      \def\ta{#1}\def\tb{#2}
      \ifx\ta\@empty{\;} \else {\ta}\fi ,
      \ifx\tb\@empty{\;} \else {\tb}\fi
      \right\rangle} 
\makeatother

\def\Sp{{\rm Sp}}

\def\rB{{\rm B}}

\def\rC{{\rm C}}
\def\rD{{\rm D}}

\def\rO{{\rm O}}

\def\det{{\rm det}}

\def\Cent{{\rm Z}}

\def\fff{\mathfrak{f}}
\def\bfff{{\overline{\mathfrak{f}}}}
\def\fgg{\mathfrak{g}}
\def\fhh{\mathfrak{h}}

\def\foo{\mathfrak{o}}
\def\fpp{\mathfrak{p}}

\def\fuu{\mathfrak{u}}

\def\fyy{\mathfrak{y}}

\def\cA{\mathscr{A}}
\def\sB{\mathscr{B}}

\def\sD{\mathscr{D}}

\def\sL{\mathscr{L}}
\def\sM{\mathscr{M}}

\def\sQ{\mathscr{Q}}

\def\sX{\mathscr{X}}

\def\cB{{\mathcal{B}}}

\def\rU{{\rm{U}}}

\def\cA{\mathcal{A}}
\def\cB{\mathcal{B}}

\def\cN{\mathcal{N}}

\def\cS{\mathcal{S}}

\def\bA{{\mathbb{A}}}

\def\bC{{\mathbb{C}}}

\def\bQ{{\mathbb{Q}}}

\def\bR{{\mathbb{R}}}
\def\bS{{\mathbb{S}}}

\def\bZ{{\mathbb{Z}}}
\def\GL{{\mathrm{GL}}}

\def\sfGKx{{\sfG(\kur)_{x}}}

\def\sfGx{{\sfG_x}}
\def\sfGxp{{\sfG'_{x'}}}

\def\fgl{{\mathfrak{gl}}}

\def\Id{\mathrm{Id}}

\def\Im{{\rm Im\,}}
\def\Hom{{\rm Hom}}
\def\End{{\rm End}}
\def\Mat{{\rm Mat}}

\def\ind{{\rm ind}}

\def\rank{\mathrm{rank}}
\def\Gal{\mathrm{Gal}}

\long\def\delete#1{}

\def\ceil#1{{\lceil #1 \rceil}}
\def\foo{{\mathfrak{o}}}
\def\sV{{\mathscr{V}}}

\def\iinn#1#2{\langle \!  \langle \, #1,#2 \, \rangle \! \rangle}

\def\Hom{{\mathrm{Hom}}}

\def\sS{{\mathscr{S}}}

\def\Stab{{\mathrm{Stab}}}
\def\Sp{{\mathrm{Sp}}}
\def\Mp{{\mathrm{Mp}}}

\def\rG{{\mathrm{G}}}

\def\kur{{k^{\text{ur}}}}

\def\Xllp{{{\mathsf{X}}_{\lambda,\lambda'}}}

\def\crho{{\check{\eta}}}
\def\Latt{{\mathrm{Latt}}}
\def\sLatt{{\mathrm{Latt}_V^\sharp}}

\def\quarter{\frac{1}{4}}
\def\half{\frac{1}{2}}

\def\fracmm{{\frac{1}{m}}}
\def\fracdmm{{\frac{1}{2m}}}
\def\fractdmm{{\frac{3}{2m}}}

\def\Fr{{\rm Fr}}

\def\dalpha{{\rm{d}\alpha}}

\def\Wy{{\mathbb{W}}}

\def\bw{{\bar{w}}}

\newcommand{\SSS}[1]{{\S#1}}

\def\zeroobj#1{\mathring{#1}}

\def\sfGE{{\zeroobj{\sfG}}} 
\def\sfGEp{{\zeroobj{\sfG'}}} 
\def\sfSLE{{\zeroobj{\mathsf{SL}}}}

\def\sfggE{{\zeroobj{\sfg}}}

\def\tsfW{{\zeroobj{\sfW}}}

\def\tlambda{{\zeroobj{\lambda}}}
\def\tlambdap{{\zeroobj{\lambda}'}}

\def\tww{{\zeroobj{w}}}

\def\tsfXll{{\zeroobj{\sfX}_{\tlambda,\tlambdap}}}
\def\tsfM{{\zeroobj{\mathsf{M}}}}

\def\tsfSl{{\zeroobj{\mathsf{S}}_\tlambda}}

\def\tbSw{{\zeroobj{\mathbb{S}}_{\tww}}}
\def\tbSl{{\zeroobj{\mathbb{S}}_{\tlambda}}}

\def\SpA{{\Sp_A}}
\def\omegaA{{\omega_A}}

\def\Bg{{B_{\fgg}}}
\def\Bgp{{B_{\fgg'}}}

\def\Vr{{V^{[r]}}}
\def\Vrhalf{{V^{[r+\half]}}}
\def\Vnr{{V^{[-r]}}} \def\sVr{{\sV^r}}

\newcommand{\sz}{-\fracdmm}

\newcommand{\sfXll}{{\sfX_{\lambda,\lambda'}}}
\newcommand{\sfSw}{{\sfS_{\bw}}}
\newcommand{\sfSl}{{\sfS_{\lambda}}}
\newcommand{\sfSlp}{{\sfS'_{\lambda'}}}
\newcommand{\bfSl}{{H_{\Gamma}}}
\newcommand{\bfSlp}{{H'_{\Gamma'}}}
\newcommand{\trSp}{{\triangle_\alpha(\sfSlp)}}

\newcommand{\tr}{\mathrm{tr}}
\newcommand{\diag}{\mathrm{diag}}

\newcommand{\bfY}{\mathbf{Y}}

\newtheorem{thm}[subsubsection]{Theorem}
\newtheorem{lemma}[subsubsection]{Lemma}
\newtheorem{prop}[subsubsection]{Proposition}
\crefname{prop}{Proposition}{Propositions}
\Crefname{prop}{Proposition}{Propositions}
\newtheorem{cor}[subsubsection]{Corollary}
\newtheorem*{prop*}{Proposition}
\theoremstyle{definition}
\newtheorem{definition}[subsubsection]{Definition}

\def\remark{\noindent{\bf Remark. }}

\def\pV{{\breve{V}}} 

\def\sE{{\mathscr{E}}}

\def\VE{{V(E)}}
\def\DE{{D(E)}}
\def\fooDE{{\foo_{\DE}}}
\def\fffDE{{\fff_{\DE}}}

\def\sLE{{\sL_E}}
\def\sLEp{{\sL'_E}}
\def\sLEz{{\sL^0_E}}
\def\sLEzp{{\sL'^0_E}}

\def\piSigma{{\pi_\Sigma}}
\def\piSigmap{{\pi'_{\Sigma'}}}

\crefname{cnd}{Condition}{Conditions}
\crefformat{cnd}{Condition~(#2#1#3)}
\crefname{case}{case}{cases}
\crefformat{case}{case~(#2#1#3)}

\newcounter{myenumi}
\def\savemyenumi{\setcounter{myenumi}{\value{enumi}}}
\def\resumemyenumi{\setcounter{enumi}{\value{myenumi}}}

\newcommand{\trivial}[2][]{\if\relax\detokenize{#1}\relax
{\color{orange} \vspace{0em} The following is trivial:  #2}
\else 
\ifx#1h\relax\else {\red Wrong argument!} \fi
\fi
}

\author{Hung Yean Loke}

\address{Hung Yean Loke, Department of Mathematics,
National University of Singapore,
Block S17, 10 Lower Kent Ridge Road, Singapore 119076}

\email{matlhy@nus.edu.sg}

\author{Jia-jun Ma}

\address{Jia-jun Ma, Unit 408, Academic Building No.1,
The Institute of Mathematical Sciences,
The Chinese University of Hong Kong,
Shatin, N.T., Hong Kong}

\email{jjma@ims.cuhk.edu.hk}

\author{Gordan Savin}

\address{Gordan Savin, Department of Mathematics
        University of Utah
        Salt Lake City, UT 84112}
\email{savin@math.utah.edu}



\title[Correspondences between epipelagic representations]{Local
  theta correspondences between epipelagic supercuspidal representations}

\begin{document}

\maketitle

\begin{abstract}
  In this paper we study the local theta correspondences between epipelagic
  supercupsidal representations of a type I classical dual pair
  $(G,G')$ over $p$-adic fields.  We show that, besides an exceptional case, an
  epipelagic supercupsidal representation $\pi$ of $\widetilde{G}$ lifts to an
  epipelagic supercupsidal representation $\pi'$ of $\widetilde{G}'$ if and only
  if the epipelagic data of $\pi$ and $\pi'$ are related by the moment maps.
\end{abstract}

\maketitle

\section{Introduction}

\subsection{} \label{sec:N}  Let $k$ be $p$-adic field with ring of 
integers $\foo_k$, prime ideal $\fpp_k = (\varpi)$ and residue field~$\fff$ of odd characteristic~$p$ where $\varpi$ is a fixed
uniformizer. Let $\bar{k}$ be its algebraic closure and let $\kur$ be  the maximal unramified extension of $k$. Let $\val \colon \bar{k} \to \bQ \cup \set{\infty}$ denote the
valuation map where $\val(\varpi) = 1$. 
We will denote an algebraic variety by a boldface letter, say $\bfH$. For an algebraic extension $E$ of $k$, we will denote its
$E$-points by $\bfH(E)$, and its $k$ points by the corresponding normal letter
$H$. 

If $\bfG$ is an algebraic group, then we let $\fgg$ be its Lie algebra
or the $k$ points of the Lie algebra, depending on the context. If
$G = \bfG(k)$ acts on a set $X$, then ${}^g x$ will denote $g\cdot x$
for $g\in G$ and $x\in X$. In order to simplify the situation, we
always assume that $p$ is  sufficiently large compared to the rank
of $G$. For an reductive algebraic group, we will fix a maximally $k$-split
torus $\bfS$, a maximally $\kur$-split torus $\bfT$ which is defined
over $k$ and contains $\bfS$. Since $\bfG(\kur)$ is always
quasi-split, we set $\bfY := \Cent_\bfG(\bfT)$ to be the fixed Cartan
subgroup of $\bfG$

\subsection{}
We recall the classification of irreducible type~I reductive dual pairs.  Let
$D$ be a division algebra over $k$ with a fixed involution $\tau$, which is
either \begin{inparaenum}[(i)]
\item  $k$,
\item  a
quadratic field extension of $k$, or
\item the quaternion algebra over $k$.
\end{inparaenum}
We continue to use $\val$ to denote the unique extension of the
valuation $\val$ from $k$ to $D$.  Let $\foo_D$ be the ring
of integers of $D$.  Let $V$ be a right $D$-module with an
$\epsilon$-Hermitian sesquilinear  form $\inn{}{}_V$ and
$G = \bfG(k) = \rU(V, \inn{}{}_V)$ be the unitary group preserving
$\inn{}{}_V$.  Similar notation applies to $G'$ with $\epsilon' = -\epsilon$.
The $k$-vector space $W = V \otimes_D V'$
has a natural symplectic form and $(G,G')$ is an irreducible type~I
reductive dual pair in $\Sp:=\Sp(W)$.

We will let $\widetilde{E}$ be the inverse image of a subgroup $E$ in the
metaplectic $\bC^\times$-cover $\Mp$ of $\Sp$. 
We fix a non-trivial additive character $\psi \colon k \rightarrow \bC^\times$ with conductor $\fpp_k$
and consider local theta correspondence $\theta$ arising from the
oscillator representation of $\Mp$ with respect to the
character $\psi$. 
For general information on theta correspondences, see \cites{Ho,Wa,GT}.
In this paper, we will investigate the theta correspondence between
epipelagic supercuspidal representations of $\tG$ and $\tG'$. In many
situations, the roles of $G$ and $G'$ are interchangeable. In such
cases, we only discuss  $G$ and extend all objects and 
notation to $G'$ implicitly by adding `primes'.

\subsection{} \label{sec:SD}
We briefly review some facts about epipelagic supercuspidal
representations. See \Cref{sec:Epi} and \cite{RY} for more
details. Let $\cB(\bfG,k)$ be the (extended) building of $G$. In the
setting of type~I dual pairs, $G$ and $G'$ have compact
centers. Therefore $\Cent(G)\subseteq G_{x}$ and $G_x = G_{[x]}$ for
any $x\in \cB(\bfG,k)$ where $[x]$ is the image of $x$ in the reduced
building.  Following Reeder-Yu~\cite{RY}, we construct a tamely
ramified irreducible supercuspidal representation $\pi_\Sigma$ of
$\tG$ in \Cref{sec:CEpi} from the data $\Sigma = (x,\lambda,\chi)$
where \crefname{SD}{SD}{SD} \crefformat{SD}{(SD#2#1#3)}
\begin{enumerate}[(SD1)]
\item \label[SD]{SD1} $x\in \cB(\bfG,k)$ is an epipelagic point of order $m$
  where $p \nmid m$ (c.f. \Cref{def:epipoint}),
\item \label[SD]{SD2}$\lambda$ is a stable vector in
  $\fgg_{x,-\fracmm}/\fgg_{x,-\fracmm^+}$ and
\item \label[SD]{SD3} $\chi$ is a character of the stabilizer
  $\sfS_\lambda$ of $\lambda$ in $G_x/G_{x,0^+}$.
\end{enumerate}
We call $\Sigma = (x,\lambda,\chi)$ an {\it epipelagic data} of order
$m$ for $G$.  The supercuspidal representation~$\pi_\Sigma$ has depth
$\frac{1}{m}$ and it is called an {\it epipelagic supercuspidal
  representation} of $\tG$. Since the cover $\tG$ depends on the dual
pair $(G,G')$, the notation $\pi_\Sigma$ only makes sense relative to
the dual pair $(G,G')$.

\subsection{} \label{S14}
Let $\Sigma = (x, \lambda, \chi)$ and
$\Sigma' = (x', \lambda', \chi')$ be epipelagic data of $G$ and $G'$
of order $m$ and $m'$ respectively. Suppose
$\theta(\pi_\Sigma) = \pi'_{\Sigma'}$.  By \cite{Pan,Pan02}, $\piSigma$ and
$\piSigmap$ have the same depth $\frac{1}{m} = \frac{1}{m'}$, i.e.
$m = m'$. It turns out that the data $(x,\lambda,\chi)$ and
$(x',\lambda',\chi')$ are related by a geometric picture which we now
briefly explain.

\def\sfMm{{\sfM_{-\fracdmm}}}
\def\sfMmp{{\sfM'_{-\fracdmm}}}
The points $x$ and $x'$ correspond to
self dual $\foo_D$-lattice functions $\sL$ and $\sL'$ in $V$ and $V'$
respectively (c.f. \cite{BT4,BS,Le,GY}). The tensor product $\sB = \sL \otimes \sL'$ is a self dual
$\foo_k$-lattice
function in $W$. The quotient  $\sfX = \sB_{-\frac{1}{2m}}/ \sB_{-\frac{1}{2m}^+}$ is an $\fff$-vector
space. In \eqref{eq:MM.s} we define the moment maps $\sfMm$ and $\sfMmp$: 
\[
\xymatrix@C=3em{
\fgg_{x,-\fracmm} & \ar[l]_<>(.5){\sfMm} \sfX \ar[r]^<>(.5){\sfMmp} & 
\fgg'_{x',-\fracmm}.
}
\]

Our first result is a refinement of special cases of \cite{Pan}.

\begin{prop}[\Cref{prop:M.w}]
Suppose $\theta(\piSigma)  = \piSigmap$. Then 
\begin{equation} \label[cnd]{cnd:M}\tag{M}
\parbox{0.8\textwidth}{
there exists a $\bar{w}  \in \sfX$ such that $\lambda = \sfMm(\bw)$ and
$\lambda' = - \sfMmp(\bw)$ }.
\end{equation}
\end{prop}

\Cref{cnd:M} imposes severe restriction to
the ranks of $G$ and $G'$:
\begin{prop}[\Cref{prop:RSP}] \label{prop:I.RSP} Suppose
  $\theta(\pi_\Sigma) =
  \pi'_{\Sigma'}$.
  Then $(\bfG,\bfG')$ or $(\bfG',\bfG)$ is one of the following types:
  
\begin{inparaenum}[(i)]
\item $(\rD_n,\rC_n)$, 
\item $(\rC_n,\rD_{n+1})$,
\item $(\rC_n,\rB_n)$,
\item $(\rA_n,\rA_n)$ or 
\item $(\rA_n,\rA_{n+1})$.
\end{inparaenum}
\end{prop}

Let $(x,\lambda)$ and $(x',\lambda')$ be parts of data for $G$ and $G'$
of order $m$ respectively satisfying (SD1) and (SD2). It turns
out that Condition~\eqref{cnd:M}, in all but one exceptional case
(see \Cref{case:E} in \Cref{sec:thetaII}), is a sufficient
condition for the epipelagic supercuspidal representation $\piSigma$
to lift to an epipelagic supercuspidal representation of $\tG'$.
For the ease of explaining in this introduction, we will omit the exceptional case.
Using \cref{cnd:M}, we will construct a group homomorphism
$\alpha \colon \sfS_{\lambda'}' \rightarrow \sfS_\lambda$ in
\Cref{lem:Xll}. We can now state a part of the main \Cref{thm:main0}.

\delete{\begin{thm} \label{thm:intro}
  Let $(G,G')$ be an irreducible type~I reductive dual pair such that
  $(\bfG,\bfG')$ has the form (i)-(v) in
  \Cref{prop:I.RSP}.  Let $(x,\lambda)$ and $(x',\lambda')$ be data of
  $G$ and $G'$ respectively of order $m$ satisfying \cref{SD1},
  \cref{SD2} and \cref{cnd:M}.  Suppose we are not in the
    exceptional Case~\eqref{case:E}. Then for every character
  $\chi$ of $\sfSl$,
\[
\theta( \piSigma) = \piSigmap
\]
where $\Sigma = (x,\lambda,\chi)$,
$\Sigma' = (x',\lambda',\chi^*\circ \alpha)$ and $\chi^*$ is the
contragredient of $\chi$. In particular the theta lift is nonzero.
\end{thm}

We state a converse of the above  theorem.

\begin{cor} \label{cor:Converse}
Suppose $\theta(\piSigma) = \piSigmap$ where $\Sigma = (x,\lambda, \chi)$ and $\Sigma' = (x', \lambda',\chi')$ are epipelagic data, and $(G,G')$ and $\Sigma$ do not belong to the exceptional case (E). Then
\begin{enumerate}[(i)]
\item Condition (M) is satisfied so that $\alpha \colon \sfS_{\lambda'}' \rightarrow \sfS_\lambda$ is well-defined and

\item $\chi' = \chi^* \circ \alpha$.
\end{enumerate}
\end{cor}}

\begin{thm} \label{thm:intro}
  Let $(G,G')$ be an irreducible type~I reductive dual pair such that
  $(\bfG,\bfG')$ has the form (i)-(v) in
  \Cref{prop:I.RSP}.  Let $(x,\lambda)$ and $(x',\lambda')$ be data of
  $G$ and $G'$ respectively of order $m$ satisfying \cref{SD1},
  \cref{SD2}. We assume that we are not in the
    exceptional \Cref{case:E}
	\begin{enumerate}[(i)]
	\item
	Suppose \cref{cnd:M} is satisfied. Then for every character
  $\chi$ of $\sfSl$,
\[
\theta( \piSigma) = \piSigmap
\]
where $\Sigma = (x,\lambda,\chi)$,
$\Sigma' = (x',\lambda',\chi^*\circ \alpha)$ and $\chi^*$ is the
contragredient of $\chi$. In particular the theta lift is nonzero.
\item Conversely, suppose $\theta(\piSigma) = \piSigmap$ where $\Sigma = (x,\lambda, \chi)$ and $\Sigma' = (x', \lambda',\chi')$ are epipelagic data. Then
\begin{enumerate}[(a)]
\item Condition (M) is satisfied so that $\alpha \colon \sfS_{\lambda'}' \rightarrow \sfS_\lambda$ is well-defined and

\item $\chi' = \chi^* \circ \alpha$.
\end{enumerate}	
\end{enumerate}
\end{thm}

\Cref{thm:main0} also contains a result for the exceptional
Case~\eqref{case:E} where $\pi_\Sigma$ lifts only for half of the
characters of~$\sfSl$. This should be compared with \cite{Moen93}.

\subsection*{Acknowledgment}
We would like to thank Wee Teck Gan and Jiu-Kang Yu for their valuable
comments. Hung Yean Loke is supported by a MOE-NUS AcRF Tier 1 grant
R-146-000-208-112.  Jia-Jun Ma is partially supported by ISF Grant 1138/10 during his postdoctoral  Fellowship at Ben Gurion University and HKRGC Grant CUHK 405213 during his postdoctoral fellowship in IMS of CUHK. Gordan Savin is supported by an NSF grant DMS-1359774.

\section{Epipelagic representations} \label{sec:Epi}

In this section we review Reeder-Yu's construction of epipelagic supercuspidal
representations (c.f. \cite{RY}). 

\subsection{} Let $k$ be a $p$-adic field as in \Cref{sec:N}.  
Let $\kur$ be the
maximal unramified extension of $k$ with residue field $\bfff$. 
We let $\Fr$ denote the Frobenius element such that
$\Gal(\kur/k) = \Gal(\bfff/\fff) = \langle \Fr \rangle$.

\subsection{Epipelagic points}\label{sec:EpiP}  Let $\bfG$ be an algebraic
group defined over $k$. Let $E$ be a tamely ramified extension of $k$. 
For $x\in \cB(\bfG,k) \subseteq \cB(\bfG,E)$, we set up some notation which will be used
in the rest of the paper. 
Let  
\begin{itemize}
\item 
$\bfG(E)_{x,r}$ ($r\geq 0$) and $\fgg(E)_{x,r}$ be the Moy-Prasad filtrations corresponding to $x$.  
\item $\bfG(E)_x = \Stab_{\bfG(E)}(x)$ and $\bfG(E)_{[x]} =
  \Stab_{\bfG(E)}([x])$ where $[x]$ is the image of $x$ in the reduced
  building;
\item  $\sfG_x(E) = \bfG(E)_x/\bfG(E)_{x,0^+}$;
\item 
$\sfG_{x,r}(E) =  \bfG(E)_{x,r:r^+} := \bfG(E)_{x,r}/\bfG(E)_{x,r^+}$, and
\item $\sfg_{x,r}(E) =  \fgg(E)_{x,r:r^+} := \fgg(E)_{x,r}/\fgg(E)_{x,r^+}$.
\end{itemize}
In order to abbreviate the notations of objects corresponding to
$G = \bfG(k)$, we let
\begin{itemize}
\item $G_{x,r} := \bfG(k)_{x,r} = \bfG(\kur)_{x,r} \cap \bfG(k)$, 
\item $\fgg_{x,r} := \fgg(k)_{x,r} = \fgg(\kur)_{x,r} \cap \fgg(k)$,
\item {$G_{x,r:r'} := \bfG(k)_{x,r}/\bfG(k)_{x,r'}$ and  $\fgg_{x,r:r'} := \fgg(k)_{x,r}/\fgg(k)_{x,r'}$ for $r < r'$.}
\end{itemize}
In order to abbreviate the notations of objects corresponding to
$\bfG(\kur)$, we let
\begin{itemize}
\item $\sfG_x := \sfG_x(\kur)$, $\sfG_{x,r} := \sfG_{x,r}(\kur)$ and
  $\sfg_{x,r} := \sfg_{x,r}(\kur)$.
\end{itemize}

The quotient space $\sfg_{x,r}$ is an $\bfff$-vector space and we
denote its dual space $\Hom_{\bfff}(\sfg_{x,r}, \bfff)$ by
$\csfg_{x,r}$. We have assumed in the introduction that $p$ is large
compared to the rank of $G$.  Then by \cite[Prop. 4.1]{AR},
$\csfg_{x,r}$ could be identified with $\sfg_{x,-r}$ via an invariant
bilinear form on $\fgg$. Since we are only treating classical groups
and $p \neq 2$, we will use a trace form defined later in
\Cref{def:Parings} in this paper.
The group $\sfG_x$ acts on $\csfg_{x,r}$. 
A vector $\lambda \in \csfg_{x,r}$ is called a \emph{stable vector} if the $\sfG_x$-orbit of $\lambda$ is Zariski closed in $\csfg_{x,r}$ and the stabilizer of $\lambda$ in $\sfG_x$ modulo $\Cent(\bfG(\kur))_{0:0^+}$ is a finite group.

\medskip

Let {$\Psi_{\kur}$} be the set of affine $\bfT(\kur)$-roots.  Suppose
$x \in \cA(\bfS, k) = \cA(\bfT,\kur)^\Fr$, i.e. $x$ is in the apartment
defined by $\bfS$.  Let $r(x)$ be the smallest positive value in
$\set{ \psi(x) | \psi \in \Psi_{\kur} }$. Then
$\bfG(\kur)_{x,0^+} = \bfG(\kur)_{x,r(x)}$.

\begin{definition} \label{def:epipoint} Let $m$ be an integer where
  $p \nmid m$.  A point $x$  in the apartment
  $\cA(\bfS,k)$ is called an {\it epipelagic point} of order $m$ if
  $r(x) = \frac{1}{m}$ and $\csfg_{x,\fracmm}$ contains a stable vector.
\end{definition}

By \cite{Gross}, $m$ is an even integer except when $\bfG(\kur)$ is split
and of type $\rA_{m-1}$. In particular~$m \geq~2$.

\def\rx{{r(x)}}

\subsection{Epipelagic supercuspidal representations} \label{sec:Epirep} 
Let $x$ be an epipelagic point
of order $m$ so that $r(x) = \fracmm$.  We fix an isomorphism of
abelian groups $c \colon G_{x,\rx:\rx^+} \rightarrow \fgg_{x,\rx:\rx^+}$.
For the classical group, we choose the isomorphism to be the one
induced by the Cayley transform $c(g) = 2(g-1)(g+1)^{-1}$ so that $g-1 \equiv
c(g) \pmod{\fgl(V)_{x,r(x)^+}}$.

\def\Jx{{G_{x,\fracmm}}}
\def\Jxp{{G_{x,\fracmm^+}}}

Let $\lambda \in \csfg_{x,\fracmm}^{\Fr}$ be an $\fff$-rational
functional on $\sfg_{x,\fracmm}$.  Then we get a character
\[
\psi_\lambda := \psi \circ \lambda \colon 
G_{x,\fracmm:\fracmm^+ } = \sfg_{x,\fracmm}(k) \cong \sfg_{x,\fracmm}^\Fr 
\xrightarrow{\ \lambda\ } \fff \xrightarrow{\ \psi\ }
\bC^\times.
\]
The inflation of $\psi_{\lambda}$ to $\Jx$ will also be denoted by
$\psi_{\lambda}$.

Let
\begin{equation} \label{eq1}
H_{x,\lambda} :=\Stab_{G_{[x]}}(\lambda), \ \  \sfA_{x,\lambda} := H_{x,\lambda}/\Jx.
\end{equation}
Then $H_{x,\lambda}$ is the stabilizer of $\psi_\lambda$ in~$G_{[x]}$.

\medskip

We will assume that $\lambda$ is $\sfG_x$-stable in
$\csfg_{x,\frac{1}{m}}$. In all the cases that we will consider in
this paper, $\sfA_{x,\lambda}$ is a finite abelian group. 
Note that the order of $\sfS_{x,\lambda}$ is prime to $p$ and $\Jx$ is pro-$p$.

\begin{prop} The group  $\sfA_{x,\lambda}$ splits in $H_{x,\lambda}$, i.e.
$H_{x,\lambda} = \sfA_{x,\lambda} \ltimes \Jx$. 
\end{prop} 

\begin{proof} We shall show that $\sfA_{x,\lambda}$ splits in $H_{x,\lambda}/G_{x,\frac{i}{m}}$ for all $i \geq 1$. For $i=1$ there is nothing to prove. 
Assume that we have constructed a splitting $s_i \colon \sfA_{x,\lambda} \rightarrow H_{x,\lambda}/G_{x,\frac{i}{m}}$. The obstruction to lift this splitting to 
$H_{x,\lambda}/G_{x,\frac{i+1}{m}}$ lies in $H^2(\sfS_{x,\lambda},G_{x,\frac{i}{m}:\frac{i+1}{m}})$. Since the order of $\sfS_{x,\lambda}$ is prime to $p$ and 
$G_{x,\frac{i}{m}:\frac{i+1}{m}}$ is an elementary $p$-group, this cohomology vanishes. Hence $s_i$ can be lifted to $s_{i+1}$ and the proposition follows 
by passing to a limit.  
\end{proof}

We extend
$\psi_\lambda$ to a character of $H_{x,\lambda}$ by setting
$\psi_\lambda$ to be trivial on $\sfS_{x,\lambda}$.  By
$H^1(\sfS_{x,\lambda},\Jx) = 0$, we know that all splittings are
conjugate up to $\Jx$-conjugation. Hence the extension $\psi_\lambda$
is unique and therefore canonical.

Let $\chi$ be a character of $\sfA_{x,\lambda}$ and let
$\pi_x(\lambda, \chi) := \ind_{H_{x,\lambda}}^G \psi_{\lambda} \otimes
\chi$.  By \cite[Prop.~5.2]{RY}, $\pi_x(\lambda,
\chi)$ is an irreducible supercuspidal representation of $G$.
We will call $(x, \lambda, \chi)$ an {\it epipelagic data} of
order $m$ and we call $\pi_x(\lambda, \chi)$ an (irreducible) {\it
  epipelagic supercuspidal representation} attached to the data. It
contains a minimal $K$-type represented by a coset
$\lambda = [\Gamma] = \Gamma + \fgg_{x,-\fracmm^+}$ in
$\fgg_{x,-\fracmm:-\fracmm^+}$ where $K = G_{x,\fracmm}$.

\begin{prop} \label{prop:MKT}
Suppose $G$ is a group appearing in a type I reductive dual pair.
\begin{asparaenum}[(i)]
\item
All unrefined minimal $K$-types of
  $\pi_x(\lambda,\chi)$ are $G$-conjugate to
  $\Gamma + \fgg_{x,-\fracmm^+}$. 
\item
 If $\pi_x(\lambda, \chi)$ and $\pi_x(\lambda',\chi')$ are isomorphic
  $G$-modules, then $\lambda$ and $\lambda'$ are in
  the same $G_x$-orbit. In addition if $\lambda = \lambda'$ then
  $\chi = \chi'$.
\end{asparaenum}
\end{prop}

\begin{proof}
  (i) Let $\Gamma + \fgg_{x,-\frac{1}{m}^+}$ represent the unrefined
  minimal $K$-type of $\pi$ as above. We will see in \Cref{lem:reg}
  later that $\Gamma \in \fgg_{x,-\fracmm}$ is a good element and
  $\bfH =\Cent_\bfG(\Gamma)$ is a torus. Let $H = \bfH(k)$ and
  let $\fhh = \Cent_{\fgg}(\Gamma)$ be its Lie algebra. Since
  $(x,\Gamma,\chi)$ is a tamely ramified supercuspidal data,
  $H/\Cent(G)$ is $k$-anisotropic (or see for example,
  \cite[Proposition~14.5]{Kim}). We are considering type I classical
  dual pairs so $\Cent(G)$ is anisotropic. Therefore $H$ is
  $k$-anisotropic. Hence $\cB(\bfH,k) = \set{x}$.

  Suppose $\Gamma_y + \fgg_{y,-\fracmm^+}$ is another unrefined
  minimal $K$-type of $\pi$ for some $y \in \cB(\bfG,k)$. 
Now we show that we can move $y$ to $x$. Since $\Gamma
  + \fgg_{x,-\frac{1}{m}^+}$ and $\Gamma_y + \fgg_{y,-\frac{1}{m}^+}$
  are minimal $K$-types of $\pi$, they are associates \cite{MP}, i.e. there
  exists a $g \in G$ such that
\begin{equation} \label{eq3}
{}^g\left( \Gamma_y + \fgg_{y,-\fracmm^+} \right) \cap \left( \Gamma
  +\fgg_{x,-\fracmm^+} \right) = \left({}^g\Gamma_{y}
  +\fgg_{gy,-\fracmm^+} \right) \cap \left( \Gamma
  +\fgg_{x,-\fracmm^+} \right)
\end{equation}
is nonempty. 

Now we apply the argument in the proof of \cite[Corollary~2.4.8]{KM}. By
\eqref{eq3}, there are $X''\in \fgg_{x,-\fracmm^+}$ and
$Y''\in \fgg_{gy,-\fracmm^+}$ such that ${}^g\Gamma_{y}+ Y'' = \Gamma + X'' $.
By \cite[Corollary 2.3.5]{KM}, there exists $h\in G_{x,0^+}$ such that
${}^h(\Gamma+X'') = \Gamma+ X'\in (\Gamma+\fgg_{x,-\fracmm^+})\cap \fhh \cap
\fgg_{hgy,-\fracmm}$
where $X' \in \fhh_{x,-\frac{1}{m}^+}$. By \cite[Lemma~2.4.7]{KM}
$hgy \in \cB(\bfH,k) = \set{x}$, i.e. $hgy = x$.  We consider the isomorphism
$\Ad((hg)^{-1}) \colon \fgg_{y,-\fracmm}/\fgg_{y,-\fracmm^+} \iso
\fgg_{x,-\fracmm}/\fgg_{x,-\fracmm^+}$.
It is now clear that the coset $[{}^{hg} \Gamma_y] = [\Gamma] = \lambda$. This
proves (i).

(ii) The last assertion of (ii) is \cite[Lemma 2.2]{RY}. Now we 
prove the first assertion. Note that
$\lambda = \Gamma + \fgg_{x,-\frac{1}{m}^+}$ and
$\lambda' = \Gamma' + \fgg_{x,-\frac{1}{m}^+}$ represent unrefined
minimal $K$-types of $\pi_x(\lambda, \chi)$ and
$\pi_x(\lambda', \chi')$ respectively. Since
$\pi_x(\lambda, \chi) \cong \pi_x(\lambda', \chi')$, the two minimal
$K$-types are associates. By the proof in (i) where $y = x$ and
$\Gamma_y = \Gamma'$, we conclude that there exists $g \in G$ such
that $gx = x$ and $g \lambda' = \lambda$. In particular $g \in G_x$. 
Hence $\lambda$ and $\lambda'$ are in the same $G_x$-orbit in
$\sfg_{x,-\fracmm}(k)$.
 \end{proof}

\section{Classical reductive dual pairs and local theta
  correspondence}

\subsection{Classical groups} \label{SS21} In this section, we will
define the classical groups which appear in the irreducible dual
pairs.

Let $D$ be a division algebra over $k$ with an involution $\tau$ in
one of the following cases:
\begin{enumerate}[(i)]
\item \label{D.k}$D = k$, $\tau$ is the identity map on $k$ and $\varpi_D = \varpi$.

\item \label{D.qd}$D$ is a quadratic extension of $k$, $\tau$ is the nontrivial
  Galois element in $\Gal(D/k)$, $\varpi_D$ is a uniformizer of $D$ so
  that $\varpi_D = \varpi$ if $D/k$ is unramified or
  $\tau(\varpi_D) = -\varpi_D$ if it is ramified.
	
\item \label{D.qt}$D$ is the quaternion algebra over $k$, $\tau$ is the
  usual involution on $D$, $\varpi_D$ is a uniformizer of $D$ such
  that $\varpi_D^2 = \varpi$.
 
\end{enumerate}
Let $\foo_D$ denote the ring of integers of $D$,
$\fpp_D = \varpi_D\foo_D$ denote its maximal prime ideal and
$\fff_D = \foo_D/\fpp_D$ denote its residue field. We set
$\val_D := \val(\varpi_D)$.

\def\rhalf{{\half \val_D}}
\def\iotag{{\iota_{\sfggE}}}

Let $V$ be a right $D$-vector space. Let $\End_D(V)$ denote the space
of $D$-linear endomorphisms of $V$ which acts on the left.  For
$\epsilon = \pm 1$, let $\inn{}{}_V \colon V\times V\rightarrow D$ be
an $\epsilon$-Hermitian sesquilinear form, i.e.
\[
\inn{v_1}{v_2}_V = \epsilon\inn{v_2}{v_1}_V^\tau \quad\text{and}\quad
\inn{v_1a_1}{v_2a_2}_V = a_1^\tau\inn{v_1}{v_2}_V a_2
\] 
for all $v_1,v_2 \in V$ and $a_1,a_2 \in D$. The $\epsilon$-Hermitian
form induces a conjugation $*\colon \End_D(V)\rightarrow \End_D(V)$
such that $\inn{gv_1}{v_2}_V = \inn{v_1}{g^* v_2}_V$ for all
$v_1, v_2 \in V$ and $g \in \End_D(V)$. Then
\begin{align*}
  G &  = \rU(V) = \rU(V, \inn{}{}_V) := \set{ g \in \End_D(V) | gg^* = \Id }
      \quad \text{and} \\ 
  \fgg & = \fuu(V) = \fuu(V, \inn{}{}_V) 
            := \set{ X \in \End_D(V) | X + X^* = 0}
\end{align*}
are a classical group and its Lie algebra.

\subsection{Irreducible reductive dual pairs of
  type~I} \label{sec:dualpair} Let $V$ be a right $D$-vector space
equipped with an $\epsilon$-Hermitian sesquilinear form $\inn{}{}_V$
and let $V'$ be a right $D$-vector space equipped with an
$\epsilon'$-Hermitian sesquilinear form $\inn{}{}_{V'}$ where $\epsilon' = - \epsilon$. Let $G$ and
$G'$ be the classical groups defined by $(V,\inn{}{}_V)$ and
$(V',\inn{}{}_{V'})$ respectively.

We view $V'$ as a left $D$-module by $a v = v a^\tau$ for all
$a \in D$ and $v \in V'$.  Let $W = V \otimes_D V'$. It is a
symplectic $k$-vector space with symplectic form $\inn{}{}$ given by
\begin{equation} \label{equ4}
\inn{v_1 \otimes v_1'}{v_2 \otimes v_2'} =
\tr_{D/k}(\inn{v_1}{v_2}_V \inn{v_1'}{v_2'}_{V'}^{\tau}).
\end{equation}
Then $G$ and $G'$ commute with each other in the
symplectic group $\Sp(W)$. We call $(G,G')$ an \emph{irreducible reductive
dual pair of type~I}.

\subsection{Lattice model} \label{sec:LM} We recall that
$\psi \colon k \rightarrow \bC^\times$ is a non-trivial additive character
with conductor $\fpp_k$.  Let $A$ be a self dual lattice in $W$, i.e. $A = \set{w \in W | \inn{w}{w'} \in \fpp_k,
  \forall w' \in A }$. The lattice model with respect to~$A$ of the
oscillator representation $\omega$ with respect to the character $\psi$ is
defined by
\[
\sS(A) = \Set{f\colon W \to \bC|\begin{array}{c} f(a+w) =
  \psi(\frac{1}{2}\inn{w}{a})f(w)\ \forall a\in A\\
  \text{$f$ locally constant, compactly supported}
\end{array}}.
\]
Let $\Mp(W)$ be
the metaplectic $\bC^\times$-covering of $\Sp(W)$ which acts on the
oscillator representation naturally by its definition.  The lattice
model with respect to $A$ gives a section
$\omegaA \colon \Sp(W) \hookrightarrow \Mp(W)$ of the natural projection
$\Mp(W) \twoheadrightarrow \Sp(W)$ (c.f. \cite{MVW,Wa}).  Let
$\SpA := \Stab_{\Sp(W)}(A) = \Set{g\in\Sp(W)|gA \subseteq A}$.  We
only describe $\omegaA(g)$ for $g \in \SpA$:
\begin{equation} \label{eq:LA} 
(\omegaA(g)f)(w) = f(g^{-1}w) \quad
  \forall g \in \SpA, f \in \sS(A) \text{ and } w \in W.
\end{equation}
The splitting $\omegaA$ does not depend on the choice of the self-dual
lattice $A$. More precisely, we have the following proposition which follows immediately from \Cref{lem:SS}.

\begin{prop} \label{prop:liftSpA}
There is a section
\[
\xymatrix{
\omega_0 \colon {\displaystyle \bigcup_{A \text{ is self-dual}}} \Sp_A
\ar@{^(->}[r]& \Mp(W)
}
\]
such that $\omega_0|_{\SpA} = \omega_A$ for every self dual lattice
$A$.
\end{prop}

\subsection{Epipelagic supercuspidal representations of covering
  groups} \label{sec:CEpi} Let $\Sigma = (x,\lambda,\chi)$ be an
epipelagic datum of order $m$. We retain the notation for the subgroup
$H_{x,\lambda} = \sfA_{x,\lambda} \ltimes \Jx$ and its character
$\psi_\lambda \otimes \chi$ in \Cref{sec:Epirep}. Since $G$ is
  a member of a type I dual pair, we recall that $G_x = G_{[x]}$ in \Cref{sec:SD}
  and $H_{x,\lambda}$ is a subgroup of $G_x$. We will show in
\Cref{sec:sfX} later that $G_x$ stabilizes a self-dual lattice $A$ in
$W$.  Then \Cref{prop:liftSpA} gives a splitting
\[\xymatrix{
\omega_0|_{G_x}\colon G_x \ar@{^(->}[r]& \tG_x
}
\]
of $\tG_x \rightarrow G_x$. We will identify $H_{x,\lambda}$ and $G_x$
as subgroups of $\tG_x$ via $\omega_0$. Let
$\rid_{\bC^\times} \colon \bC^\times \rightarrow \bC^\times$ be the identity map.
Under this splitting, $\tG_x = G_x \times \bC^\times$ with $\bC^\times$ acting
on the oscillator representation $\sS$ via $\rid_{\bC^\times}$.  Now
\[
\pi_\Sigma :=\ind_{H_{x,\lambda} \times \bC^\times}^{\tG} ((\psi_{\lambda} \otimes
  \chi) \boxtimes \rid_{\bC^\times})
\]
is an irreducible supercuspidal representation of $\tG$ which is also
denoted by $\pi_\Sigma^{\tG}$ or $\pi_x^{\tG}(\lambda, \chi)$.  We
will also call $\pi_\Sigma$ an epipelagic supercuspidal representation
attached to the epipelagic data $\Sigma$.

By \Cref{sec:wplus} the splitting of $G_{x,0^+}$ is canonically
defined for any $x\in \cB(\bfG,k)$. In particular, it still makes
sense to talk about positive depth minimal $K$-types. In addition
\Cref{prop:MKT} holds if we replace $\pi_x(\lambda, \chi)$ with
$\pi_\Sigma^{\tG}$ without any modification.

\section{Bruhat-Tits Buildings and Moy-Prasad filtrations of classical
  groups}\label{sec:BT}

In this section we recall some known facts about the Bruhat-Tits
buildings of classical groups. Our references are \cite{BT3,BT4,BS,Le}.

\subsection{Lattice functions} 
A (right) $\foo_D$-lattice $L$ in a right $D$-vector space $V$ is a
right $\foo_D$-submodule such that $L\otimes_{\foo_D} D = V$.

\begin{definition} \label{def:DLattic}
\begin{asparaenum}[1.]
\item Let $\Latt_V$ be the set of $\foo_D$-lattice valued functions
  $s\mapsto \sL_s$ on $\bR$ such that
\begin{inparaenum}[(i)]
\item $\sL_s \supseteq \sL_t$ if $s<t$, 
\item $\sL_{s+\val_D} = \sL_s\varpi_D$ and 
\item $\sL_s = \bigcap_{t<s} \sL_t$.
\end{inparaenum}
\item We set $\sL_{r^+} := \bigcup_{t>r} \sL_t$ and $\Jump(\sL) = \set{ r \in \bR | \sL_{r} \supsetneq \sL_{r^+}
}$.
\item Given any lattice function $\sL$, we define
\begin{align*}
  \fgl(V)_{\sL,r} :=& \Set{X\in \fgl(V)| X \sL_s \subseteq \sL_{s+r}, \forall s
                      \in \bR }  & & \forall r \in \bR, \\
  \GL(V)_{\sL,r} := & \Set{g\in \GL(V)| (g-1) \sL_s \subseteq \sL_{s+r},
                   \forall s \in \bR} & & \forall
                                          r>0, \\
  \GL(V)_{\sL} :=&  \Set{g\in \GL(V)|g\sL_s\subseteq \sL_s}.
\end{align*}
\item For $r < s$, we denote $\sL_{r:s} = \sL_r/\sL_s$.
\end{asparaenum}
\end{definition}

\begin{definition} \label{def:DNorm}
\begin{asparaenum}[1.]
\item A $D$-norm of $V$ is a function $l \colon V \rightarrow \bR
\cup \set{ \infty }$ such that for all $x, y \in V$ and $d \in D$, (i)
$l(xd) = l(x) + \val(d)$, (ii) $l(x+y) \geq \min(l(x), l(y))$ and (iii)~$l(x) = \infty$ if and only if $x = 0$. 

\item The norm $l$ is called {\it splittable} if there is a $D$-basis
  $\set{ e_i | i \in I }$ of $V$ such that
  $l(\sum_{i \in I} e_i d_i) = \inf_{i \in I}(l(e_i) + \nu(d_i))$.  Let
  $\cS \cN(V)$ denote the splittable norms on $V$.  In this paper, all
  norms refer to splittable $D$-norms.
\end{asparaenum}
\end{definition}

There is a natural bijection between $\cS \cN(V)$ and $\Latt_V$ given by
$l \mapsto ( \sL_r = l^{-1}([r,+\infty)))$. Then $\Jump(\sL)$ is the
image of $l$.

The following theorem is well known and follows directly from the definition of 
Moy-Prasad filtration \cite{MP}.

\begin{thm} \label{thm:BT3.MP} The (extended) building $\cB(\GL(V))$
  could be identified with $\Latt_V$ as $\GL(V)$-sets. This
  identification is unique up to translation
  (c.f. \cite[Theorem~2.11]{BT3}).  Suppose $x\in \cB(\GL(V))$
  corresponds to the lattice function $\sL \in \Latt_V$. Then
\begin{asparaenum}[(a)]
\item 
$\fgl(V)_{x,r} = \fgl(V)_{\sL,r}$ for $r\in \bR$, 
\item $\GL(V)_{x,r} = \GL(V)_{\sL,r}$ for $r>0$ and 
\item $\GL(V)_\sL = \GL(V)_x$.
\end{asparaenum}
\end{thm}

For the rest of this paper, we will freely interchange the notion of
points in the building of $\GL(V)$, $D$-norms and lattice functions.

\subsection{Tensor products} \label{S42} Suppose $l$ and $l'$ are two
norms on $D$-modules $V$ and $V'$.  Then there is an induced norm on
$W :=V \otimes_D V'$ such that
$(l\otimes l') (v\otimes v') = l(v)+l'(v')$ (c.f.~\cite[\S~1.11]{BT3}).  Let $\sL$ and $\sL'$ be the corresponding lattice
functions. We denote by $\sL\otimes \sL'$ the corresponding
$\foo_k$-lattice function on $V\otimes_D V'$ where
\[
(\sL\otimes \sL')_t = \sum_{r+r'=t} \sL_r\otimes_{\foo_D} \sL'_{r'}.
\]  
It is easy to see that 
\begin{equation} \label{eqJLL}
\Jump(\sL\otimes \sL') = \Jump(\sL)+\Jump(\sL').
\end{equation}

The norms $l$ and $l'$ also induce a natural norm $\Hom(l,l')$ on
$\Hom_D(V,V')$ whose corresponding lattice function is
\begin{equation} \label{eqhomr}
(\Hom(\sL,\sL'))_r := \set{w\in \Hom_D(V,V')|w(\sL_s)\subseteq \sL_{s+r}' \,
\forall s\in \bR}.
\end{equation}
In particular, every norm $l$ on $V$ defines a dual norm
$l^*:=\Hom(l,\val)$ on $V^*:= \Hom_D(V,D)$.

Under the isomorphism $\Hom_D(V,V') \cong V' \otimes_D V^*$, the
norms $\Hom(l,l')$ and $l' \otimes l^*$ coincide.  If $V = V'$, then
the Moy-Prasad lattice function $r\mapsto \fgl(V)_r$ defined in
\Cref{def:DLattic} is the tensor product lattice function
$\sL\otimes \sL^*$ on $\fgl(V) = \End_D(V)$.

\subsection{Self-dual lattice functions} \label{S43} Let $V$ be a
space with a non-degenerate sesquilinear form $\inn{}{}_V$.  

\begin{definition} \label{D421}
\begin{asparaenum}[1.]
 \item For a lattice $L$ in V, we set
\[
L^\sharp := \set{v \in V | \inn{v}{v'}_V \in \fpp_D,
  \forall v'\in L }.
\]
A lattice $L$ is called {\it self-dual} if $L = L^\sharp$. A lattice $L$ is called {\it good} if
  $L^\sharp \fpp_D \subseteq L \subseteq L^\sharp$.

\item For a lattice function $\sL$ we define its dual lattice function
  $\sL^\sharp$ by $(\sL^{\sharp})_s = (\sL_{(-s)^+})^\sharp$. If $l$
  is the norm corresponding to $\sL$, then we denote the norm
  corresponding to $\sL^\sharp$ by~$l^\sharp$. If we identify $V$ with
  $V^*$ using the form $\inn{}{}_V$, then the norm $l^*$ on $V^*$
  translates to the norm $l^\sharp$ on $V$.

\item A lattice function $\sL$ is called {\it self-dual} if and only
  if $\sL = \sL^\sharp$. In terms of norm, it is equivalent to
  $l^\sharp = l$ (c.f. \cite[Prop.~3.3]{BS}) and we say that $l$ is
  self-dual. Let $\sLatt$ be the set of self-dual lattice
  functions. Clearly $\sLatt$ is the $\sharp$-fixed point set of
  $\Latt_V$.

\item When $\sL$ is self-dual, we define
  $\fgg_{\sL,r} := \fgg\cap \fgl(V)_{\sL,r}$,
  $G_{\sL,r} := G\cap \GL(V)_{\sL,r}$, $G_{\sL} := G \cap \GL(V)_{\sL}$,
  $\sfG_\sL := G_{\sL}/G_{\sL,0^+}$ and
  $\fgg_{\sL,r:s} := \fgg_{\sL,r}/\fgg_{\sL,s}$.
\end{asparaenum}
\end{definition}

\noindent \remark  If we identify $V^*\otimes_D V'^*$ as $(V\otimes_D V')^*$, then by a
calculation on a splitting basis, we have
$(l\otimes l')^* = l^* \otimes l'^*$
(c.f. \cite[(18),(21), Sect. 1.12]{BT3}).  In particular, suppose that $V$
and $V'$ are formed spaces, and $\sL$ and $\sL'$ are self-dual lattice
functions. It is easy to see that
$(l_\sL\otimes l_{\sL'})^\sharp = l_\sL \otimes l_{\sL'}$, i.e. it is
self-dual. Hence $\sL\otimes \sL'$ is a self-dual lattice on
$V \otimes_D V'$. 

\subsection{}
We recall that $k$ is a $p$-adic field with $p\neq 2$.  For a
classical group $G$ defined over~$k$, $\cB(\bfG,k)$ could be
identified canonically with the set of splittable self-dual norms on
$V$ (c.f. \cite{BT4,GY}).  The following theorem is the culmination of
\cite{BT4}, \cite{BS}, \cite{Le} and \cite{GY}.

\begin{thm}\label{thm:SB2}
\begin{enumerate}[(i)]
\item There is a natural $G$-equivariant bijection between $\cB(\bfG,k)$ and
$\sLatt$.

\item Suppose $x\in \cB(\bfG,k)$ corresponds to $\sL \in \sLatt$. Then 
\begin{enumerate}[(a)]
\item $\fgg_{\sL,r} = \fgg_{x,r}$ for $r\in \bR$,
\item $G_{\sL,r} = G_{x,r}$ for $r>0$ and
\item $G_{\sL} = G_x$.
\end{enumerate}
\end{enumerate}
\end{thm}

\subsection{}
Let $r \in \Jump(\sL)$ so that
$\sfL_r := \sL_{r:r^+} = \sL_r/ \sL_{r^+}$ which is nonzero. The
sesquilinear form $\inn{}{}_V$ induces a nonzero pairing
$\sL_{r}\times \sL_{-r} \rightarrow \foo_D$ and a non-degenerate
pairing over~$\fff_D$:
\[
\sfL_r \times \sfL_{-r} \rightarrow \fff_D.
\]
In particular we have 
\[
\Jump(\sL) = -\Jump(\sL). 
\]

The structure of $\sfG_{\sL}$ is described in the following lemma. It
is well known so we omit its proof. Also see \Cref{app:split}

\begin{lemma} \label{lem:G0}
  Let $\nu_D = \val(\varpi_D)$.
  Then
  \begin{equation}
    \label{eq:sfG}
    \sfG_{\sL} \cong \sfG_0 \times \sfG_{\rhalf} \times \prod_{r\in
      \Jump(\sL)\cap (0,\rhalf)} \GL(\sfL_r)
  \end{equation}
  where $\sfG_0\cong \rU(\sfL_0)$ and
$\sfG_\rhalf \cong \rU(\sfL_{\rhalf})$ where $\sfL_{\rhalf}$ is
equipped with the form
$\inn{[v_1]}{[v_2]} = \inn{v_1}{v_2\varpi_D^{-1}}_V
\pmod{\fpp_D}$. \qed
\end{lemma}

\section{Tame base changes and epipelagic points} 

\subsection{} \label{S51} Let $D$ be the division algebra over $k$ as in \Cref{SS21}.  
Let $V$ be a $D$-module with an $\epsilon$-Hermitian sesquilinear form $\inn{}{}_V$ and let $\bfG$ be the classical group defined over $k$ such that $G := \bfG(k) = \rU(V,\inn{}{}_V)$. Suppose $E$ is a tamely ramified
finite extension of $k$ or $\kur$ such that $\bfG$ splits.  In this
section we study the relations between buildings under tamely ramified
field extensions.

In all cases, it is standard to construct an $E$-vector space $\pV$
obtained by certain base change of $V$ so that
$\bfG(k) \subseteq \bfG(E)$ are subgroups of $\GL_E(\pV)$. By
\cite{GY} there is a canonical bijection
$\cB(\bfG,k) \iso
\cB(\bfG,E)^{\Gal(E/k)}$.
The next proposition describes this bijection in terms of splittable
norms on $V$ and $\pV$. 

\begin{prop} \label{lem:TBC} 
We identify buildings of classical groups with the corresponding set of splittable norms.
\begin{asparaenum}[(i)]
\item Suppose $D = k$. Let $\pV = V\otimes_k E$ and let
  $i_V \colon V \rightarrow \pV$ be given by $v\mapsto v\otimes 1$. Let
  $\inn{}{}_\pV$ be the $E$-linear extension of $\inn{}{}_V$. Then 
  $\bfG(E) = \rU(\pV,\inn{}{}_\pV)$. The bijection
  $\cB(\bfG,k) \iso \cB(\bfG,E)^{\Gal(E/k)}$
  is given by $l_V \mapsto l_V \otimes_k (\val|_E)$ and its inverse map is
  $l_\pV \mapsto l_\pV \circ i_V$.

\item Suppose $D$ is a quadratic extension of $k$. We fix a field
  embedding $\iota \in \Hom_k(D,E)$ and view $D$ as a subfield of
  $E$.  Let $\pV = V\otimes_D E$ and let $i_V : V \rightarrow \pV$ be
  given by $v\mapsto v\otimes 1$. Then $\bfG(E) \cong  \GL_E(\pV)$.  The bijection
  $\cB(\bfG,k) \iso \cB(\bfG,E)^{\Gal(E/k)}$
  is given by $l_V \mapsto l_V\otimes_D (\val|_E)$ and its inverse map is
  $l_\pV \mapsto l_\pV \circ i_V$.

\item Suppose $D$ is the quaternion  algebra over $k$.  We fix a
  subfield $L$ of $E$ which is a quadratic extension of $k$. We identify $L$ with a subfield of $D$
  and fix a $d \in D$ such that $d^2 \in k^\times$, $d^\tau = - d$ and
  $\Ad(d)$ acts on $L$ by the non-trivial Galois action. Let
  $\pr \colon D \rightarrow L$ be the projection of $D = L\oplus Ld$. Then
  $Q(v_1,v_2):=\pr(\inn{v_1 d}{v_2}_V)$ defines an $L$-bilinear from
  on $V$.  Let $\pV = V\otimes_L E$ and let $\inn{}{}_{\pV}$ be the
  $(-\epsilon)$-symmetric $E$-linear extension of $Q$. Then {$\bfG(E) \cong  \rU(\pV,
  \inn{}{}_{\pV})$.} The bijection
  $\cB(\bfG,k) \iso \cB(\bfG,E)^{\Gal(E/k)}$
  is given by $l_V \mapsto l_V\otimes_L (\val|_E)$ and its inverse map is
  $l_\pV \mapsto l_\pV \circ i_V$.
\end{asparaenum}
\end{prop}

Before we give the proof of \Cref{lem:TBC}, we first recall the
uniqueness result stated in \cite[\SSS{1.2}]{GY}.

\begin{lemma} \label{lem:BBC}
  Let $\cB$ and $\cB'$ be two $G$-sets satisfying the axioms of building of
  $\bfG$ over $k$ (See \cite[\SSS{2.1}]{Tits} and
\cite[\SSS{1.9.1}]{GY}.).  Let $j\colon \cB \to \cB'$ be a bijection such
  that
\begin{enumerate}[(i)]
\item $j$ is $G$-equivariant, i.e. $j(g\cdot x) = g \cdot j(x)$ for all
  $g \in G$, and
\item its restriction to an apartment $\cA$ is affine.
\end{enumerate}
Then $j$ is unique up to the translation by an element in 
$X_*(\Cent(G)^\circ) \otimes \bR$. 
\end{lemma}
In our cases, $\Cent(G)$ is anisotropic so the map $j$ is unique. 

\begin{proof}[Proof of \Cref{lem:TBC}] 
\begin{asparaenum}[(i)]
\item {
We consider the following diagram:
\[
\xymatrix{
\cB(\bfG,k) \ar@{==}[r] \ar@{^(->}[d]& \cB(\bfG,E)^{\Gal(E/k)} \ar@{^(->}[d] \\
\cB(\GL_k(V)) \ar@{=}[r] & \cB(\GL_E(\pV))^{\Gal(E/k)}.
}
\]
The buildings in the top row are the fixed point sets of the involutions $\sharp$ of the buildings in the bottom row. The bottom map is
$l_V \mapsto l_V \otimes (\val|_E)$.  It is a $\GL_k(V)$-invariant  map. It is a bijection, since it suffices to check this on an apartment, where it is obvious. The map 
sends self-dual norms to self-dual norms, hence it induces the top row isomorphism, by restriction. It is the canonical 
isomorphism $\cB(\bfG,k) \iso\cB(\bfG,E)^{\Gal(E/k)}$ by \Cref{lem:BBC}. 
This proves~(i). }

\item 
We refer to the computation in \cite[\SSS{1.13}]{BT3}.  
Let $\iota_1, \iota_2$ be two $k$-embeddings  of $D$ into $E$ (so $\iota$ is one of the two). 
Let $\pV^{\iota_i} = V\otimes_{D,\iota_i} E$. 
Then $V\otimes_k E\cong \pV^{\iota_1} \oplus \pV^{\iota_2}$, so we have a natural action of  $\Gal(E/k)$  on $\pV^{\iota_1} \oplus \pV^{\iota_2}$. 
Now Part (ii) follows by applying a similar argument as in (i) to the following diagram:
\[
\xymatrix{
\cB(\bfG,k) \ar@{==}[r] \ar@{^(->}[d]& \cB(\bfG,E)^{\Gal(E/k)} \ar@{^(->}[d] \\
\cB(\GL_D(V)) \ar@{=}[r] & \cB(\GL_E(\pV^{\iota_1})\times \GL_E(\pV^{\iota_2}))^{\Gal(E/k)}.
}
\]

\item  Note that
  $H(v_1,v_2) = \pr(\inn{v_1}{v_2}_V)$ defines a Hermitian form on
  $V$. Moreover, $\rG(k) = \rU(V,Q) \cap \rU(V,H)$. 
Part (iii) follows by applying a similar argument as in (i) to the following diagram:
\[
\xymatrix{
  \cB(\bfG,k) \ar@{==}[r] \ar@{^(->}[d]& \cB(\rU(\pV,
  \inn{}{}_{\pV}))^{\Gal(E/k)} \ar@{^(->}[d] \\ 
  \cB(\rU(V,H)) \ar@{=}[r] & \cB(\GL_L(\pV))^{\Gal(E/k)}. }
\]
\end{asparaenum}
\end{proof}

For an $\foo_D$-lattice function $\sL$ corresponding to
$x \in \cB(\bfG,k)$, we will denote by $\sL^E$ the $\foo_E$-lattice
function in $\pV$ corresponding to $x \in \cB(\bfG,E)^{\Gal(E/k)}$ in
the above proposition. We need the following application of \Cref{lem:BBC}
in our study of the epipelagic points.

\begin{lemma} \label{L621} Let $\sL$ be the self-dual lattice function
  corresponding to a point $x$ in $\cB(\bfG,k)$ of order $m$. Suppose
  $\bfG$ splits under a tamely ramified extension $E$ with
  ramification index $m$.  Then $\Jump(\sL)$ is contained in either
  $\frac{1}{m} \bZ$ or $\frac{1}{2 m} + \frac{1}{m} \bZ$.
\end{lemma}

\begin{proof}
  We have $\val(E) = \fracmm \bZ$.  Let $\sL^E$ be the
  $\foo_E$-lattice function. Let $J = \Jump(\sL)$ and
  $J^E=\Jump(\sL^E)$.  By \cite[\SSS{4.2}]{RY}, $\sL^E$ corresponds to
  a hyperspecial point in $\bfG(E)$. Hence $J^E = j_0 + \fracmm \bZ$
  for some $j_0\in [0,\fracmm)$.
  
  By \Cref{lem:TBC}, we have $J \subseteq J^E = J+\fracmm
  \bZ$. Since $\sL$ is self-dual, $J = -J$. Hence 
  $-j_0+\fracmm \bZ =-J^E \subseteq - J +
  \fracmm\bZ \subseteq J +\fracmm \bZ \subseteq J^E + \fracmm \bZ
  = j_0 +\fracmm \bZ$. Therefore $j_0 = 0$ or $\fracdmm$. 
  The lemma follows. 
\end{proof}

\subsection{Epipelagic points} \label{sec:EP} 
 Let $x \in \cB(\bfG,k)$ be an epipelagic point of order $m$.
We recall that $\kur$ is
the maximal unramified extension of $k$. Let~$E$ be the totally
ramified extension of~$\kur$ of degree~$m$. We fix a uniformizer
$\varpi_E$ such that $\varpi_E^m = \varpi$. Let
$\Gal(E/\kur) = \langle \sigma \rangle$ where
$\sigma(\varpi_E) = \zeta \varpi_E$ and $\zeta$ is a primitive $m$-th
root of unity.  Let $\Fr \in \Gal(E/k)$ denote the lift of the Frobenius
automorphism in $\Gal(\kur/k)$ such that $\Fr(\varpi_E) = \varpi_E$. Now
$\Gal(E/k) = \braket{\Fr, \sigma}$.

The group $\bfG$ splits over $E$ \cite[\SSS{4.1}]{RY}.  We recall \Cref{sec:N}
that $\bfT$ is a maximally $\kur$-split torus in $\bfG$ containing $\bfS$
and defined over $k$.  Let $\bfY = \Cent_{\bfG}(\bfT)$ be a Cartan
subgroup of $\bfG$.  Then $x$ is a hyperspecial point in
$\cA(\bfY,E)$. We have $\bfG(E)_x^\sigma = \bfG(\kur)_x$,
$\bfG(E)_{x,r}^\sigma = \bfG(\kur)_{x,r}$ for $r > 0$, and
$\varpi_E^d \fgg(E)_{x,0} = \fgg(E)_{x,\frac{d}{m}}$.  The building
$\cB(\bfG,\kur)$ embeds into $\cB(\bfG,E)$ as the $\sigma$-invariant set
and
\[
\cB(\bfG,k)  = \cB(\bfG,\kur)^\Fr = (\cB(\bfG,E)^\sigma)^\Fr.
\]

We set $D(E) = D\otimes_k E$. We equip $D(E)$ with the tensor product norm
  of valuations of $D$ and $E$. Let $\sE$ and
$\sD$ be the lattice functions in $E$ and $D$ respectively defined by their valuations.  Then
$\sD\otimes \sE$ is the corresponding lattice function on $\DE$. Let
$\foo_{D(E)} = (\sD\otimes \sE)_0$, 
$\fpp_{D(E)} = (\sD\otimes \sE)_{0^+}$ and
$\fff_{D(E)} = \foo_{D(E)}/\fpp_{D(E)}$ which is a semisimple algebra over $\fff_E$.

Let $\sL$ be the self-dual lattice function in $V$ corresponding to
$x \in \cB(\bfG,k)$. Let $\sL_E := \sL \otimes \sE$ be the self-dual
$\foo_k$-lattice function in $V(E) := V \otimes_k E$. In fact it is an
$\fooDE$-lattice function. We have following situations.
\begin{enumerate}[(i)]
\item If $D = k$, then {$\sL_E = \sL^E$.}
\item If $D$ is a quadratic extension of $k$, then
  $D(E) \cong E\times E$, $V(E) = \pV^{\iota_1}\oplus \pV^{\iota_2}$,
  $\sL_E = \sL^{\iota_1, E} \oplus \sL^{\iota_2, E}$ where
  $\pV^{\iota_1}$ and $\pV^{\iota_2}$, $\sL^{\iota_1,E}$ and
  $\sL^{\iota_2,E}$ correspond to the two different $k$-embeddings of
  $D$ into $E$.
	
\item Suppose $D$ is the quaternion algebra over $k$. We fix a quadratic
  extension $L$ of $k$ in $E$. Then $D(E) \cong \Mat_2(E)$,
  $V(E) = \pV^{\iota_1}\oplus \pV^{\iota_2}$ and
  $\sL_E = \sL^{\iota_1, E} \oplus \sL^{\iota_2, E}$ where
  $\sL^{\iota_1,E}$ and $\sL^{\iota_2,E}$ are $\foo_E$-lattice
  functions corresponding to the two different $k$-embeddings of $L$
  into $E$.
\end{enumerate}
Clearly, $\Jump(\sL_E) = \Jump(\sL^{\iota_i, E}) = \Jump(\sL) + \fracmm\bZ$. The
Galois group $\Gal(E/k)$ acts on $\VE$ by $s(v \otimes x) = v \otimes s(x)$ for
$s \in \Gal(E/k)$, $v \in V$ and $x \in E$. For $g\in \bfG(E)$
and $s \in \Gal(E/k)$, we have $s(g) = s \circ g\circ s^{-1}$ as $\DE$-linear
automorphism on $\VE$.  In Cases (ii) and (iii), under the
decomposition, $\bfG(E)$ acts diagonally on
$V(E) = \pV^{\iota_1}\oplus \pV^{\iota_2}$.

Extending the notation in  \Cref{def:DLattic}, 
we have $\bfG(E)_x = \bfG(E)_{\sL_E}$, $\bfG(E)_{x,r} = \bfG(E)_{\sL_E,r}$ and
$\fgg(E)_{x,r} = \fgg(E)_{\sL_E,r}$ by \Cref{lem:TBC}.

\subsection{Kac-Vinberg gradings} \label{sec:Kac-Vinberg}
{In \cite[\SSS{4}]{RY}, Reeder and Yu connect the Moy-Prasad filtration at an epipelagic point with the Kac-Vinberg gradings of Lie algebras over the residue fields. 
We review their results here.}

By the classification of hyperspecial points for split classical groups (see Remark in \Cref{sec:apartment}), we can pick a point $x_0 \in \cA(\bfT, \kur)$ following the recipe in \cite[Section 3.2]{RY} such that $\Jump(\sLEz) = \Jump(\sLE)$ where $\sLEz$ is the lattice function in $V(E)$ corresponding to $x_0$.
The action of the generator $\sigma$ of $\Gal(E/\kur)$ on the Cartan subgroup $\bfY$ induces an action $\vartheta$ on $X_*$. 
Then $x = x_0 + \frac{1}{m}\crho$ where $\crho \in X_*^\vartheta$. 
Let $t:=\crho(\varpi_E) \in \bfG(E)$.
The $\foo_{D(E)}$-lattice function corresponding to $x_0$ is  $\sLEz = t^{-1}
\sLE$. 

We have isomorphisms
\[
  \bfG(E)_{x_0} \xrightarrow{\Ad(t)} \bfG(E)_x \quad \text{and} \quad
  \fgg(E)_{x_0,0} \xrightarrow{\Ad(t)} \fgg(E)_{x,0}\xrightarrow{\varpi_E^j}
  \fgg(E)_{x,\frac{j}{m}}.
\]
Let $\sfGE = \bfG(E)_{x_0}/ \bfG(E)_{x_0,0^+}$ and
$\sfggE = \fgg(E)_{x_0,0}/ \fgg(E)_{x_0,0^+}$. 
Let $\vartheta$ be the  automorphisms  on $\sfG(E)_{x_0}$ and $\sfg(E)_{x_0}$ induced by the $\sigma$ actions
on $\bfG(E)_{x_0}$ and $ \fgg(E)_{x_0,0}$ respectively. 
Let $\theta :=\Ad(t^{-1})\circ \sigma \circ \Ad(t)$  be the automorphisms on
$\sfGE$ and on $\sfggE$ induced by the $\sigma$  actions
on $\bfG(E)_x$ and $\fgg(E)_{x,0}$.  
Let $\sfggE^{\theta,\zeta^{-j}}$
be the $\zeta^{-j}$-eigenspace of $\theta$ on $\sfggE$. Then
\begin{enumerate}[(a)]
\item $\theta = \Ad(\tt) \vartheta$ where $\tt =
  t^{-1} t^\sigma = \crho(\zeta) \pmod{G(E)_{x_0,0^+}}$;

\item
  $\Ad(t) \colon \sfGE^\theta \iso \sfG_x := \bfG(\kur)_{x}/\bfG(\kur)_{x,0^+}$
  is an isomorphism and

\item
  $ \varpi_E^j\ad(t) \colon \sfggE^{\theta,\zeta^{-j}} \iso
  \sfg_{x,\frac{j}{m}}$
  is $\sfGE^\theta$-equivariant with $\sfGE^\theta$ acting on the right hand
  side via~(b). Here we recall $\sfg_{x,\frac{j}{m}}$ in \Cref{sec:EpiP}.
\end{enumerate}
By putting $j = -1$ in (c), we define
$\iotag :=\varpi_E^{-1} \ad(t)\colon \sfggE^{\theta, \zeta} \iso \sfg_{x,-\fracmm}$.
Then $\iotag$ is a bijection between the set of stable vectors for the
$\sfGE^\theta$ action on $\sfggE^{\theta,\zeta}$ and the set of stable
vectors for the action $\sfGKx$ on $\sfg_{x,-\fracmm}$. The former
was studied by Vinberg \cite{Vin} and Levy \cite{Levy}.

\section{Moment maps}\label{sec:mm}

\subsection{} \label{sec:MM.gen} Let $W = V \otimes_D V'$. Using the
sesquilinear forms, we define $\Psi\colon W \iso \Hom_D(V,V')$ and
$\Psi' \colon W \iso \Hom_D(V',V)$ by
\[
\Psi(v \otimes v') (v_1) = v' \inn{v}{v_1}_{V} \mbox{ and } \Psi'(v
\otimes v') (v_1') = v \inn{v'}{v_1'}_{V'}
\]
for all $v, v_1 \in V$ and $v', v_1' \in V'$.  Now $g\in G$ and
$g'\in G'$ acts on $\Hom_D(V',V)$ by the formula
$(g,g')\cdot w = g'w g^{-1}$.

\begin{definition}\label{def:Parings}
\begin{asparaenum}[1.]

\item We define a non-degenerate $G$-invariant symmetric $k$-bilinear
  form\footnote{We warn that our trace form $\Bg$ has a factor of $\frac{1}{2}$.} $\Bg \colon \fgg\times \fgg\to k$ by $\Bg(X_1,X_2) = \half
  \tr_{D/k}\tr(X_2^*X_1)$.

\item 
We define an operator $\star\colon \Hom_D(V,V') \rightarrow
\Hom_D(V',V)$ by
 \[
\inn{w(v)}{v'}_{V'} = \inn{v}{w^\star(v')}_V \quad \forall
w\in \Hom_D(V,V'), v\in V, v'\in V'.
\]
We note that if $x \in W$ and $\Psi(x) = w$, then $\Psi'(x) = w^\star$.

\item We define the moment map $M\colon W \cong \Hom_D(V,V') \to \fgg$ 
and $M'\colon W\to \fgg'$ by  
\[
M(w) = w^\star w \mbox{ and } M'(w) = w w^\star.
\] 
\end{asparaenum}
\end{definition}

By definition $M$ and $M'$ are $G \times G'$-equivariant. 

\begin{lemma}\label{lem:MPsi}
Suppose $w_1,w_2, w\in W$, $X\in \fgg$ and $X'\in \fgg'$. Then
\begin{asparaenum}[(a)]
\item $\inn{w_1}{w_2}  = \tr_{D/k}\tr(w_2^\star w_1)$,
\item $\inn{X\cdot w}{w} = 2 \Bg(M(w),X)$ and $\inn{X'\cdot
    w}{w} = 2 \Bgp(-M'(w),X')$.
\end{asparaenum}
\end{lemma}

The proof is a straightforward computation using \eqref{equ4} and the
definition of $\star$. We will leave it to the reader.

\subsection{}
Let $\sL$ and $\sL'$ be two self-dual lattice functions on $V$ and
$V'$ respectively.  Let $\sB =\sL\otimes \sL'$ on $W=V\otimes_D V'$.

\begin{lemma} \label{lem:MB} 
\begin{enumerate}[(i)]
\item We have $\Jump(\sB) = \Jump(\sL) + \Jump(\sL')$.
\item The lattice function $\sB$ is self-dual in $W$,
  i.e. $\sB_r^\sharp = \sB_{-r^+}$.
\item Under the isomorphism $\Psi\colon W \iso \Hom_D(V,V')$, 
\[
 \Psi(\sB_r) = \Set{w\in \Hom_D(V,V')| w\sL_s \subseteq \sL'_{s+r}\ \forall s\in \bR}.
\]
\item We have $(\Psi(\sB_r))^\star = \Psi'(\sB_r)$.   
\item \label{lem:MB.M} We have $M(\sB_r) \subseteq \fgg_{\sL,2r}$ and
  $M'(\sB_r) \subseteq \fgg'_{\sL',2r}$. 
\end{enumerate}
\end{lemma}

\begin{proof}
Part (i) is \Cref{eqJLL}. Part (ii) is explained in the Remark in \Cref{S43}. 
By~\eqref{eqhomr} the right hand side of (iii) is the lattice function on $\Hom_D(V,V') = V^* \otimes_D V'$. 
On the other hand $\Psi$ maps $V \otimes V'$ to $V^* \otimes V'$. 
We have seen in \Cref{S43} that the norm $l^*$ on $V^*$ translates to the norm $l^\sharp = l$ on $V$.  It follows that the lattice function on the right hand side of (iii) corresponds to  the lattice function $\sB$ under $\Psi$. This proves~(iii).  If $w \in W$ then $\Psi'(w) = \Psi(w)^\star$. This proves (iv). 
  Part (v) follows directly from (iii) and~(iv).
\end{proof}

\def\bfGsL{{\bfG_{\sL}}}
\def\bfGsLp{{\bfG'_{\sL'}}}

\subsection{}
We could view $\sB_s$, $\fgg_{\sL,2s}$ and $\fgg'_{\sL',2s}$ as
schemes over $\foo_k$. Since $\star$ is $\foo_k$-linear, the moment
maps defined over the generic fibers as in \Cref{sec:MM.gen} extend to
morphisms between these $\foo_k$-schemes. The $\foo_k$-group scheme
$\bfGsL\times \bfGsLp$ acts on all these objects and the moment maps
are equivariant maps.

Let $\sfW_s = \sB_s/\sB_{s^+}$,
$\sfg_{\sL,2s} = \fgg_{\sL,2s}/\fgg_{\sL,2s^+}$ and
$\sfg'_{\sL',2s} = \fgg'_{\sL',2s}/\fgg'_{\sL',2s^+}$.  We get
morphisms, as certain quotients of the moment maps over the special
fiber,
\begin{equation} \label{eq:MM.s} 
\sfM_s \colon \sfW_s \rightarrow \sfg_{\sL,2s}
  \quad \text{and} \quad \sfM'_s \colon \sfW_s \rightarrow \sfg'_{\sL',2s}.
\end{equation}
The actions of $\bfGsL\times \bfGsLp$ reduce to
$\sfG_{\sL}\times \sfG'_{\sL'}$ actions on $\sfW_s$, $\sfg_{\sL,2s}$
and $\sfg_{\sL',2s}'$. These are the moment maps over the residual
field $\fff$ which we will study later.

\section{Theta correspondences I} \label{sec:thetaI}

\subsection{} \label{sec:bfX}
In this section we let $(G,G')$ be a reductive dual pair
in $\Sp(W)$ as in \Cref{sec:dualpair}. 

Let $\Sigma = (x,\lambda, \chi)$ and $\Sigma' = (x',\lambda',\chi')$
be epipelagic supercuspidal data for $G$ and $G'$ respectively.  Let
$\piSigma = \pi_x^{\tG}(\lambda,\chi)$ (resp.
$\piSigmap = \pi_{x'}^{\tG'}(\lambda',\chi')$) be the corresponding epipelagic
representation of $\tG$ (resp. $\tG'$).  From now on, we assume
$\theta(\piSigma) = \piSigmap$. As discussed in \Cref{S14}, $x$ and $x'$ are
both epipelagic points of order $m$ for some $m\geq 2$.

Let $\sL$ and $\sL'$ be self dual lattice functions corresponding to $x$ and
$x'$ respectively. Let $\sB = \sL \otimes \sL'$. Let
$\sfX = \sfW_{-\fracdmm} = \sB_{-\fracdmm}/ \sB_{-\fracdmm^+}$. We denote the
moment maps defined in \eqref{eq:MM.s} by
$\sfM \colon \sfX \rightarrow \sfg_{\sL,-\fracmm}$ and
$\sfM' \colon \sfX \rightarrow \sfg'_{\sL',-\fracmm}$.

\begin{prop} \label{prop:M.w} Suppose that
  $\theta(\piSigma) = \piSigmap$. Then there exists a
  $\bar{w} \in \sfX$ such that $ \lambda = \sfM(\bar{w}) $ and
  $\lambda' = -\sfM'(\bar{w})$.
\end{prop}

\begin{proof}
  By \cite{Pan} there exists
  $(y,y') \in \cB(\bfG,k)\times \cB(\bfG',k)$,
  $\sB' = \sL_y \otimes \sL_{y'}'$ and $w \in \sB_{-\frac{1}{2m}}'$
  such that $M(w)+\fgg_{y,-\fracmm^+}$ is an unrefined minimal
  $K$-type of $\piSigma$, and $-M'(w)+\fgg'_{y',-\fracmm^+}$ is an
  unrefined minimal $K'$-type of $\piSigmap$. The moment maps $M$ and $M'$
  commute with the action of $G \times G'$-conjugation.  By
  \Cref{prop:MKT}  and conjugating by $G \times G'$, we may assume
  that $y = x$, $y' = x'$, $w \in \sB_{-\frac{1}{2m}}$,
  $M(w) + \fgg_{x,-\frac{1}{m}^+} = \sfM(\bar{w}) = \lambda$ and
  $-M'(w)+\fgg_{x',-\frac{1}{m}^+}' = -\sfM'(\bar{w}) = \lambda'$
  where $\bar{w} = w + \sB_{-\frac{1}{2m}^+} \in \sfX$.
\end{proof}

\begin{cor} \label{cor:JP} We have $\Jump(\sB) \subseteq \frac{1}{2m} +
  \frac{1}{m} \bZ$. Moreover either
  \begin{enumerate}[(i)] 
  \item $\Jump(\sL) \subseteq \frac{1}{m} \bZ$ and $\Jump(\sL')
    \subseteq \frac{1}{2m} + \frac{1}{m} \bZ$ or \item $\Jump(\sL)
    \subseteq \frac{1}{2m} + \frac{1}{m} \bZ$ and $\Jump(\sL')
    \subseteq \frac{1}{m} \bZ$.
  \end{enumerate}
\end{cor}

\begin{proof}
  By \Cref{L621}, $\Jump(\sL)$ (resp. $\Jump(\sL')$) is a subset of
  $\frac{1}{m} \bZ$ or $\frac{1}{2m} + \frac{1}{m} \bZ$.  From the
  proof of the last proposition, $\bar{w}$ is a nonzero element in
  $\sB_{-\frac{1}{2m}}/\sB_{-\frac{1}{2m}^+}$. In particular
  $-\frac{1}{2m} \in \Jump(\sB) = \Jump(\sL) + \Jump(\sL')$. The
  corollary follows.
\end{proof}

\subsection{} \label{sec:MVinberg} 
Reeder and Yu connect the Moy-Prasad filtration at an epipelagic point with the Kac-Vinberg grading of Lie algebras over the residue fields, as 
described previously in \Cref{sec:Kac-Vinberg}.  Now we relate this with the moment maps.
 
Let $\sLEz = t^{-1} \sLE$ and $\sLEzp = {t'}^{-1} \sLEp$ denote
the $\fooDE$-lattice functions corresponding to $x_0$ and $x_0'$ respectively as
in \Cref{sec:EP}. By \Cref{cor:JP} we are in one of the following two cases.
\begin{enumerate}[(a)]
\item We have $\Jump(\sLEz) = \fracmm\bZ$ and
  $\Jump(\sLEzp) = \fracdmm + \fracmm \bZ$. In this case we set
  $\sfV :=\sLEz_{,0:0^+}$ and $\sfV' :=\sLEzp_{,-\fracdmm:-\fracdmm^+}$.

\item We have $\Jump(\sLEz) = \fracdmm + \fracmm\bZ$ and
  $\Jump(\sLEzp) = \fracmm \bZ$. In this case we set
  $\sfV :=\sLEz_{,-\fracdmm:-\fracdmm^+}$ and $\sfV' :=\sL'^0_{E,0:0^+}$.
\end{enumerate}
Both $\sfV$ and $\sfV'$ are $\fffDE$-modules.  In Case (a), we
assign a non-degenerate Hermitian forms $\sfV$ by
$\iinn{}{}_{\sfV} := \inn{}{}_V \pmod{\fpp_{D(E)}}$ and a non-degenerate
Hermitian form on $\iinn{}{}_{\sfV'} := \inn{}{}_{V'}\varpi_E \pmod{\fpp_{D(E)}}$.
In Case (b), the bilinear forms are defined similarly.  
Thus $\sfGE = \rU(\sfV)$ (resp. $\sfGE' = \rU(\sfV')$) as $\fff_{\DE}$-linear transformations on $\sfV$ (resp. $\sfV'$)
preserving the form.

Let $\sfL^0(E)_r = \sL_{E,r:r^+}^0$.
The actions $\sigma$ and $t^{-1} \circ \sigma \circ t$ on
$\sL_{E,r}^0$ induce actions on $\sfL^0(E)_r$ which we denote by
$\vartheta$ and $\theta$ respectively. It is compatible with the
$\vartheta$ and $\theta$ actions on $\sfGE$ defined in \Cref{sec:Kac-Vinberg} in the sense that for
$g\in \sfGE$ we have
$\vartheta(g) = \vartheta \circ g \circ \vartheta^{-1}$ and
$\theta(g) = \theta \circ g\circ \theta^{-1}$ as linear
transformations on $\sfL^0(E)_r$.

Let $\sB(E) = \sB \otimes_{\foo_k} \sE$ and
$\sB^0(E) = (t^{-1}, t'^{-1}) \sB(E) = \sLEz \otimes_{\foo_{D(E)}}
\sLEzp$.
\def\sfMz{\sfM^0}
Let $\sfW^0_{-\fracdmm} = \sB^0(E)_{-\fracdmm}/\sB^0(E)_{-\fracdmm^+}$ and 
$\tsfW := \sfV \otimes_{\fffDE} \sfV' \cong \Hom_{\fffDE}(\sfV,\sfV')$.  We
observe the following diagram:
\begin{equation} \label{eq:tMM}
\vcenter{
\xymatrix@C=3em{
\tsfW \ar[d]_{\tsfM} &\ar[l]_\sim \sfW^0_{-\fracdmm} \ar[d]^{\sfMz} 
& \ar[l]
\sB^0(E)_{-\fracdmm} \ar[d]^{M} \ar[r]_{\sim}^{(t,t')}  
& 
\sB(E)_{-\fracdmm} \ar[d]^{M} \\
\sfggE & \sfg^0_{-\fracmm} \ar[l]_{\varpi_E}^\sim
& \ar[l]
\fgg(E)_{x_0,-\fracmm} 
\ar[r]_<>(.5){\sim}^<>(.5){\Ad(t)}  
& 
\fgg(E)_{x,-\fracmm}. 
}}
\end{equation}

In the above diagram, the
$\bfG(E)_{x}\times \bfG'(E)_{x'}$-equivariant map $M$ on the far right
translates to the $\sfGE \times \sfGEp$-equivariant map $\tsfM$ on the
far left.
Here $\sfMz$ is defined by \eqref{eq:MM.s} with respect to lattices $\sLEz$ and
$\sLEzp$, and  $s=-\fracdmm$.
The explicit formula for the map $\tsfM$ is exactly the same as that for $M$ with respect to the forms
$\iinn{}{}_{\sfV}$ and $\iinn{}{}_{\sfV'}$. We remark that
$\sfW^0_{-\fracdmm}$ is not equipped with any sesquilinear form.  On
the other hand $\tsfW$ has isomorphic vector space structure as
$\sfW^0_{-\fracdmm}$ but it is equipped with a tensor product form.

\subsection{The ranks of $G$ and $G'$} 
Let $r$ be the rank of $\fgg$ and let $P = P(X)$ be the coefficient of
$z^r$ in $\det(z {\mathrm{I}}_\fgg +\ad X)$. Then $P$ is an $\Ad(G)$-invariant homogeneous
rational function on $\fgg$ defined over $k$ such that the set of
regular semisimple elements in $\fgg$ is $P^{-1}(\bA-\set{0})$  where $\bA$ is the
  affine line.

  The following lemma is a consequence of \cite[Lemma~13]{Gross}.

\begin{lemma}\label{lem:reg}
  Let $\lambda$ be a stable vector in
  $\fgg_{x,-\frac{1}{m}:-\frac{1}{m}^+}$ and
  $\gamma\in \fgg_{x,-\frac{1}{m}}$ be a lifting of $\lambda$. Then
  $\gamma$ is a regular semisimple element. Let
  $\bfT = \Cent_\bfG(\gamma)$. Then $\gamma$ is a good element of depth
  $-\fracdmm$ with respect to $\bfT$, i.e. for every root $\alpha$ of
  $\bfG(\barkk)$ with respect to $\bfT(\barkk)$, $\dalpha(\gamma)$ is
  nonzero and $\val(\dalpha(\gamma)) = -\fracmm$.
\end{lemma}

\begin{proof}
  Let $f(X) = \varpi_E \Ad(t)(X)$ and let $\barff$ be the induced map
  in the following diagram:
\[
\xymatrix@R=1em{
 \gamma \ar@{|->}[d] \ar@{}[r]|<>(.5){\in}&\fgg_{x,-\fracmm}\ar[d]
 \ar@{}[r]|*{\subset} &\fgg(E)_{x,-\fracmm} \ar@{->>}[d] \ar[r]^{f} &\fgg(E)_{x_0,0} \ar@{->>}[d]\\
\lambda \ar@{}[r]|<>(.5){\in}&\fgg_{x,-\frac{1}{m};-\frac{1}{m}^+}\ar@{}[r]|*{\subset}&\sfg_{x,-\fracmm}(E) \ar[r]^<>(.5){\barff} & \sfg_{x_0}(E)\makebox[0em][l]{$=\sfggE$.}
}
\] 
Since $\lambda$ is a stable vector, $\barff(\lambda)$ is a regular
semisimple element in $\sfggE$ and
$[P(f(\gamma))]\neq 0 \in \fff_E$. Therefore, $P(\gamma) \neq 0$ and
$\gamma$ is a regular semisimple element in $\fgg$.

Next we prove that $\gamma$ is good. Let
$R(\bfG(\barkk),\bfT(\barkk))$ be the set of roots. Let
$T'(\barkk) := \Ad(t)T(\barkk)$.  Then
$\alpha \mapsto \alpha' := \alpha\circ \Ad(t^{-1})$ gives a map
$R(\bfG(\barkk),\bfT(\barkk)) \rightarrow
R(\bfG(\barkk),\bfT'(\barkk))$.
Moreover $\alpha'$ reduces to a root $\overline{\alpha'}\in \sfggE$.
For any $\alpha \in R(\bfG(\barkk),\bfT(\barkk))$,
$\overline{d\alpha'}([f(\gamma)]) \neq 0$ since
$[f(\gamma)] = \bar{f}(\lambda)$ is regular semisimple in
$\sfggE$. This implies that $\val(\dalpha(\gamma)) = -\fracmm$ which
proves the lemma.
\end{proof}

\begin{prop}\label{prop:RSP}
Suppose $\theta(\piSigma) =\piSigmap$. 
Then $(\bfG,\bfG')$ or $(\bfG',\bfG)$ is one of the following types:
\begin{inparaenum}[(i)]
\item $(\rD_n,\rC_n)$, 
\item $(\rC_n,\rD_{n+1})$,
\item $(\rC_n,\rB_n)$,
\item $(\rA_n,\rA_n)$, 
\item $(\rA_n,\rA_{n+1})$.
\end{inparaenum}
\end{prop}

\begin{proof}
  By \Cref{prop:M.w} there exists $w \in \sB_{-\fracdmm} \subset W$ such that
  $[M(w)] = \sfM(\bar{w}) \in \fgg_{x,-\frac{1}{m}:-\frac{1}{m}^+}$
  and
  $[M'(w)] = \sfM'(\bar{w}) \in
  \fgg'_{x',-\frac{1}{m}:-\frac{1}{m}^+}$
  are stable vectors. By \Cref{lem:reg}, both $M(w)$ and $M'(w)$ are
  regular semisimple elements.

  Now we show that (i) to (v) list all possible cases which satisfy
  the following condition.
  \begin{equation}\label{equ17}
    \text{There is a
      $w\in W$ such that $M(w)$ and $M'(w)$ are both regular semisimple.} 
  \end{equation}

  We may base change to the algebraic closure $\barkk$ so that $G = \bfG(\barkk)$
  and $G' = \bfG'(\barkk)$ are split groups. We fix a maximal (split) torus $Y$
  and identify its Lie algebra $\fyy$ with $\bA^{\dim\fyy}$. We list
  $P|_\fyy = \prod_{\alpha\in \Phi(G,Y)}d\alpha$ explicitly when
  $G$ is one the following groups.
  \begin{description}
  \item[$\GL(n)$] $\fyy = \bA^n$,
    $P(a_1,\cdots, a_n) = \prod_{i \neq j} (a_i-a_j)$,
  \item[$\Sp(2n)$] $\fyy = \bA^n$,
    $P(a_1,\cdots, a_n) = \prod_{i\neq j}(a_i^2- a_j^2)\prod_j
    (-4a_j^2)$,
  \item[$\rO(2n)$] $\fyy = \bA^n$, 
    $P(a_1,\cdots, a_n) = \prod_{i\neq j}(a_i^2-a_j^2)$,
  \item[$\rO(2n+1)$] $\fyy = \bA^n$, 
    $P(a_1,\cdots, a_n) = \prod_{i\neq j}(a_i^2-a_j^2)\prod_j (-a_j^2)$.
  \end{description}

  By the classification of dual pairs, we only have to consider one
  of the following reductive dual pairs:
  \begin{enumerate}[(a)]
  \item $(\rO(2n),\Sp(2n'))$,
  \item $(\Sp(2n), \rO(2n'+1))$ and
  \item $(\GL(n),\GL(n'))$.
  \end{enumerate}
  We may assume that
  ${\mathrm{rank}} \, G \leq {\mathrm{rank}} \, G'$.
  Let $\Wy$ and $\Wy'$ be the Weyl groups of $G$ and $G'$ 
    with respect to $Y$ and $Y'$
  respectively.  By the first and second fundamental theorems of classical invariant theory $M$ induces an isomorphism  $W/G \cong \fgg$ (see \cite{Howe95, GW}) which in turn induces an isomorphism $W/G \times G' \cong \fgg/G'$.  We get following
  diagram:
  \[
  \xymatrix{
  \fyy \ar@{^(->}[r] \ar[d] & \fgg \ar[d]& \ar@{->>}[l]_{M} W \ar[r]^{M'} \ar[d]& \fgg'\ar[d] &
  \fyy' \ar@{_(->}[l] \ar[d] \\
  \fyy/\Wy \ar[r]^{\sim}&  \fgg/G & \ar[l]_<>(.5){\sim} W/ G \times G'
  \ar[r]& \fgg'/G' &\ar@{_(->}[l]_{\sim} \fyy'/\Wy'. 
  }
  \]

  We remind the readers that in the lemma below all the vector spaces
  and algebraic groups are defined over $\barkk$.

  \begin{lemma} \label{LJ1} Suppose
    ${\mathrm{rank}} \, G \leq {\mathrm{rank}} \, G'$.
    Then there is a vector subspace $\AAA$ of $W$ which is stable under
    the action of $Y \times Y'$ and such that the following
    diagram below commutes.
\begin{equation} \label{eq:cd}
\vcenter{
\xymatrix{ \fyy \ar@{->>}[d] & & \AAA / Y \times Y'
  \ar@{_(->>}[ll]_<>(.5){\sim}
  \ar@{^(->}[rr]  \ar@{->>}[d] & & \fyy' \ar@{->>}[d] \\
  \fyy / \Wy \ar[r]^{\sim} & \fgg/ G & W/ G \times G'
  \ar@{_(->}[l]_<>(.5){\sim} \ar@{^(->}[r] &
  \fgg/ G' & \fyy'/\Wy' \ar[l]_{\sim}.}
}
\end{equation}
\end{lemma}

The proof of the lemma is given in \Cref{app:M0}.

\medskip

The top row of \eqref{eq:cd} defines an inclusion map
$\Upsilon_\fyy \colon \fyy\hookrightarrow
\fyy'$.
This map is the natural inclusion
$\bA^{\rank G} \hookrightarrow \bA^{\rank G'}$ which is well known to
the experts (for example see~\cite{Ad}).

Using the bottom row in \eqref{eq:cd}, we have an inclusion
$\fyy/\Wy \hookrightarrow \fyy'/\Wy'$ induced by $\Upsilon_\fyy$.  Let
$P$ and $P'$ be the invariant polynomials for $G$ and $G'$ defined
before \Cref{lem:reg}. Then \eqref{equ17} is equivalent to the
following statement:
\[
\text{There is an $X\in \fyy$ such that $P|_\fyy(X) \neq 0$ and
  $P'|_{\fyy'}(\Upsilon_\fyy(X)) \neq 0$.}
  \]
 It follows by inspection that (i) to (v) are all the possible cases. 
\end{proof}

\section{Theta correspondences II} \label{sec:thetaII} 
In this section  we study theta correspondences of epipelagic representations. 
By \Cref{prop:RSP}, \Cref{prop:M.w} and \Cref{cor:JP}, it is enough to consider the following situations:
\begin{enumerate}[(C1)]
\item  \label{it:C0} The dual pair $(\bfG,\bfG')$ is one of the
  following types:
\begin{inparaenum}[(i)]
\item $(\rD_n,\rC_n)$, 
\item $(\rC_n,\rD_{n+1})$,
\item $(\rC_n,\rB_n)$,
\item $(\rA_n,\rA_n)$, 
\item $(\rA_n,\rA_{n+1})$.
\end{inparaenum}

\item \label{it:C1} The points $x \in \cB(\bfG,k)$ and $x' \in \cB(\bfG',k)$
  are epipelagic points of order $m$. In particular, $m\geq 2$. Let
  $\sL$ and~$\sL'$ denote the corresponding $\foo_D$-lattice
  functions.

\item \label{it:C2} $\sB = \sL \otimes \sL'$ where
  $\Jump(\sB) \subseteq \frac{1}{2m} + \frac{1}{m} \bZ$.

\item \label{it:C3} There exists a
  $\bw \in \sfX = \sB_{-\fracdmm} / \sB_{\fracdmm}$ such that
  $\sfM(\bw) = \lambda \in \sfg_{x,-\fracmm}$ and
  $-\sfM'(\bw) = \lambda' \in \sfg'_{x',-\fracmm}$ are stable vectors.
\end{enumerate}

\subsection{} 
We now study the geometry of $\sfX$ and the moment maps which will
eventually determine the local theta correspondences.

By (C4) $\lambda$ is a stable vector so
$\tlambda := \iotag^{-1}(\lambda)$ is a regular semisimple element in
$\sfggE$ (c.f. \Cref{sec:EP}).  When
$\tlambda\in \sfggE\subseteq \End_{\fff_\DE}(\sfV)$ is not of full
rank, it requires a special treatment. By (C1) and the classification
of epipelagic points in \cite{RY,Gross}, this is exactly in the
following situation:
\begin{equation}
  \label[case]{case:E} \tag{E}  \parbox[c][][c]{0.8\textwidth}{ 
    $D$ is a ramified quadratic extension of $k$, the dual pair
    $(G, G')$ is a pair of unitary groups of the same rank $n$, and
    $\tlambda \in \sfggE$ has rank $n-1$.}
\end{equation}

\def\bSw{{\bS_\bw}}
\def\bSl{{\bS_\lambda}}

Define 
\begin{align*}
\sfXll & := \sfM^{-1}(\lambda)\cap \sfM'^{-1}(-\lambda') \text{ and } \\
\sfSw & := \Stab_{\sfSl \times \sfSlp}(\bw).  
\end{align*}
We note that $\sfS_\lambda$ is abelian in all our cases so all its
irreducible representations are one dimensional characters.

\begin{lemma} \label{lem:Xll}
\begin{asparaenum}[(i)]
\item \label{lem:Xll.1} The set $\sfXll$
  is the $\sfSl$-orbit of $\bw$ in $\sfX$.

\item There is a group homomorphism $\alpha\colon \sfSlp \rightarrow \sfSl$ such
  that $\trSp := \set{ (\alpha(g'),g') | g' \in \sfSlp }$ is a subgroup of
  $\sfS_{\bw}$.

\item If we are not in \Cref{case:E}, then $\sfSl$ acts freely on
  $\sfXll$ and $\sfS_{\bw} = \trSp$.

\item In \Cref{case:E}, let $\bSw = \Stab_{\sfSl}(\bw)$ and
  $\bSl := \set{g\in \sfSl| g\circ \lambda = \lambda }$\footnote{Let
    $h \in G_x$ and $\gamma\in \fgg_{x,-\fracmm}$ be any lifts of $g$ and
    $\lambda$ respectively. We consider $h$ and $\gamma$ as elements in $\Hom_k(V,V)$. 
		Then $g\circ \lambda := h\circ \gamma + \fgg_{x,-\fracmm^+}\in
    \fgg_{x,-\fracmm:-\fracmm^+}$
    is well defined.}. Then $\bSw = \bSl$ and so the character $\chi$
  of $\sfS_{\lambda}$ occurs in $\bC[\sfXll]$ if and only if
  $\chi|_{\bSl}$ is trivial.
\end{asparaenum}
\end{lemma}

The proof is given in \Cref{app:Xll}.

\medskip

\noindent \remark\ 
\begin{inparaenum}
\item The homomorphism $\alpha$ induces a map
  $\alpha^* \colon \widehat{\sfS_{\lambda}} \rightarrow
  \widehat{\sfS_{\lambda'}'}$
  given by $\alpha^*(\chi) = \chi \circ \alpha$. The definition
  depends on the choice of $\bw\in \sfXll$. On the other hand, it is
  well-defined up to conjugation by the proof in \Cref{app:Xll}. Hence
  the map between the Grothendieck groups induced by $\alpha^*$ is
  independent of the choice of $\bw$.

\item In the exceptional \Cref{case:E}, \Cref{lem:Xll}~(iv) will lead to the
  fact that not all epipelagic representations can occur in this local theta
  correspondence.  The extreme case is the well known fact that not all
  characters of $\rU(1)$ occur in the oscillator representation
  of~$\Mp(2)$(see~\cite{Moen93}).
\end{inparaenum}

\begin{thm}\label{thm:main0}
Suppose (C1) to  (C4) hold.
For any
  character $\chi$ of $\sfS_{\lambda}$, let $\Sigma = (x,\lambda, \chi)$ and 
  $\Sigma' = (x',\lambda', \chi')$ where $\chi' = \chi^*\circ \alpha$ and 
  $\chi^*$ is the contragredient representation of $\chi$. 
\begin{enumerate}[(i)]
\item Suppose we are not in the exceptional \cref{case:E}. Then 
\begin{equation} \label{eq:CEpi}
\theta( \piSigma) = \piSigmap.
\end{equation}
In particular the theta lift is nonzero.

\item Suppose we are in the exceptional \Cref{case:E}. Then \eqref{eq:CEpi}
  holds for $\chi \in \widehat{\sfSl}$ such that $\chi|_{\bSl}$ is
  trivial. 
\end{enumerate}
\end{thm}

The above theorem gives \Cref{thm:intro} (i).

\subsection{} \label{sec:sfX}
We set $B = \sB_{-\fracdmm}$ and $A = \sB_{\fracdmm}$. By (C3),
$A^\sharp = A$ and $B^\sharp = \sB_{\fracdmm^+} = \sB_{\fractdmm}$.
Since $m\geq 2$, $\sB_{-\fracdmm}\fpp_k \subseteq \sB_{\fractdmm}$ and
$B^\sharp$ is a good lattice. We form the following exact sequences of $\fff$-vector
spaces: 
\begin{equation} \label{equ9}
\xymatrix{ 0\ar[r]& \sfY \ar[r]& \sfW \ar[r] & \sfX \ar[r] & 0}
\end{equation}
where
\[
\sfW = B/B^\sharp = \sB_{-\fracdmm:\fractdmm},\ \
\sfX = B/A = \sB_{-\fracdmm:\fracdmm} \text{ \ and \ }
\sfY = A/B^\sharp = \sB_{\fracdmm:\fractdmm}.
\]
The symplectic form on $W$ induces a non-degenerate $\fff$-symplectic form on $\sfW$ and
$\sfY$ is a maximal isotropic subspace in $\sfW$. Let
\begin{align*}
\sfP & = (G_{\sL}/G_{\sL,\frac{2}{m}}) \times
(G'_{\sL'}/G'_{\sL',\frac{2}{m}}) \text{ and } \\
\sfJ & = (G_{\sL}/G_{\sL,\fracmm}) \times
(G'_{\sL'}/G'_{\sL',\fracmm})=\sfGx\times \sfGxp.
\end{align*}
Then \eqref{equ9} is an exact sequence of $\sfP$-modules.
 The proof of following lemma is given in \Cref{app:split}.

\begin{lemma} \label{lem:split}
The natural quotient $\sfP \twoheadrightarrow \sfJ$ has a splitting such that
the exact sequence \eqref{equ9} splits as 
$\sfJ$-modules. We denote the splitting by $\sfW = \sfY\oplus \sfX$. 
\end{lemma}

\subsection{Proof of \Cref{thm:main0}} 
We recall the lattice model $\sS(A)$ in \Cref{sec:LM}.  Let $\sS(A)_B$
be the subspace of functions in $\sS(A)$ with support in $B$.  For
$f\in \sS(A)_B$, $w\in B$ and $b'\in B^\sharp$,
$f(w+b') = \psi(\half\inn{w}{b'}) f(w) = f(w)$. Therefore, we could
view $\sS(A)_B$ as a subspace in $\bC[\sfW]$. We fix the splitting
$\sfW = \sfY\oplus \sfX$ in \Cref{lem:split}.  Let $\sfJ$ act on
$\bC[\sfX]$ by translation.  Since
$f(w+a) = \psi(\frac{1}{2}\inn{w}{a})f(w)$ for all $a \in A$, the
restriction map $R_\sfX$ from $\sfW$ to $\sfX$ induces a $\sfJ$-module
isomorphism
\begin{equation} \label{eq:RX}
\xymatrix{
R_{\sfX} : \sS(A)_B \ar[r]^{\ \ \ \ \ \sim} & \bC[\sfX]
}
\end{equation}
whose inverse map $R_\sfX^{-1}$ is given by
\[
(R_\sfX^{-1} F)(w) = \psi \left( \half \inn{x}{y} \right) F(x)
\]
for all $w\in B$ such that $w \equiv x+y \pmod{B^\sharp}$ with
$x + B^\sharp \in \sfX$ and $y + B^\sharp \in \sfY$.

\medskip

For $f\in \sS(A)_B$, $w\in B$ and $g\in G_{\sL,\fracmm}$, we have $(g^{-1}-1)w \in
A$. By \eqref{eq:LA}  and \Cref{lem:MPsi}, we have\footnote{See also the proof of \cite[Theorem~5.5]{Pan}} 
\begin{equation} \label{eq:gf} 
\begin{split}
  \omegaA(g) f(w) & = f(g^{-1} w) 
  = f((g^{-1}-1)w+ w)  \\
   & =  \psi\left(\half\inn{w}{(g^{-1}-1)w} \right)f(w) = 
\psi\left(\frac{1}{2}\inn{(g-1)w}{w} \right)f(w)  \\
   & = \psi(\Bg(M(w),c(g)))f(w) = \psi_{\sfM(\bw)}(g)f(w)
 \end{split}
\end{equation}
where $\bw$ is the image of $w$ in $\sfX = B/A$. 
Similarly $\omegaA(g')f(w) = \psi_{-\sfM'(\bw)}(g')f(w)$ for all $g'\in G_{\sL,\fracmm}$.

Let $\sS(A)_B^\lambda$ be the
subspace of functions in $\sS(A)_B$ such that $G_{\sL,\fracmm}$ acts by
$\psi_\lambda$. Then it follows from~\eqref{eq:gf} and \eqref{eq:RX} that
\[
\sS(A)_B^\lambda = R_\sfX^{-1}(\bC[\sfM^{-1}(\lambda)]).
\]
A similar consideration applies to $\lambda' \in \sfg_{\sL',-\fracmm}'$ too.
Let
$\sS^{\lambda, \lambda'} := \sS(A)_B^{\lambda} \cap
\sS(A)_B^{\lambda'}$.  Then
\[
\sS^{\lambda,\lambda'} = R_\sfX^{-1}(\bC[\Xllp]).
\]

We fix a $\barww$ in $\Xllp$. By \Cref{lem:Xll}, $\sfSl \twoheadrightarrow
\sfXll$ given by $s\mapsto s\cdot \bw$ is a surjection
of $\sfSl\times \sfSlp$-set. 
Here $(s,s') \in \sfSl\times \sfSlp$ acts on $\sfSl$ by $(s,s') \cdot s_0 = s s_0 \alpha(s')^{-1}$ for all $s_0\in \sfSl$.

By the decomposition of regular representation of $\sfS_{\lambda}$, we have
\begin{equation}\label{eq:dec}
\bC[\Xllp]  \subseteq \bC[\sfSl]= \bigoplus_{\chi\in \widehat{\sfS_{\lambda}}}
\bC_{\chi}\otimes \bC_{\chi^*\circ \alpha}
\end{equation}
as $\sfS_{\lambda}\times \sfS'_{\lambda'}$-modules. In (i),
\eqref{eq:dec} is an equality.  In (ii) the summand
$\bC_{\chi}\otimes \bC_{\chi^*\circ \alpha}$ occurs in $\bC[\sfXll]$
if and only if $\chi|_{\bSl}$ is trivial by \Cref{lem:Xll}.

Fix any $\sfS_\lambda$-character $\chi$ which occurs in $\bC[\sfXll]$. It is
clear that the  $H_{\lambda}$ and $H'_{\lambda'}$ in
\eqref{eq1} act on the space
$R_\sfX^{-1}(\bC_{\lambda}\otimes \bC_{\chi^*\circ \alpha})$ by the
characters $\psi_\lambda \otimes \chi$ and
$\psi_{\lambda'} \otimes \chi^*\circ \alpha$ respectively.
By Frobenius reciprocity, the subspace $R_\sfX^{-1}(\bC_{\lambda}\otimes \bC_{\chi^*\circ \alpha}) \subseteq \sS(A)$
induces a non-zero intertwining map
\[
\xymatrix{
\pi_x^{\tG}(\lambda,\chi) \boxtimes
\pi_{x'}^{\tG'}(\lambda',\chi^*\circ \alpha) \ar[r]& \sS(A).}
\]
The left hand side is irreducible so the above
map is an injection. Since the left hand side is also supercuspidal, by the smoothness of $\sS(A)$, we conclude that the
left hand side is a direct summand in $\sS(A)$ and we have a
projection map from $\sS(A)$ to the left hand side. 
This completes the proof of \Cref{thm:main0}. \qed

\subsection{Proof of \Cref{thm:intro}~(ii)} \label{sec:Converse}
Part (a) is a restatement of \Cref{prop:M.w}. By \Cref{thm:intro} $\piSigma$ has a nonzero theta lift and $\theta(\piSigma) = \pi'_{\Sigma''}$ where $\Sigma'' = (x',\lambda',\chi^* \circ \alpha)$. By the uniqueness of the theta lift \cite{GT,Wa}, $\piSigmap = \pi'_{\Sigma''}$ and by  \Cref{prop:MKT} $\chi' = \chi^* \circ \alpha$. This proves (b).
\qed

\appendix

\section{Proof of \Cref{lem:split}} \label{app:split}
 \def\vD{{\val_D}} 

\subsection{} \label{SA1}
First the fact that $\sfP\rightarrow \sfJ$ splits follows from the work of McNinch \cite{Mc}. For our case, the splitting could be constructed by an elementary method which we will explain 
below. 

 We retain the notation in \Cref{SS21}. Let $K$ be the maximal unramified
   extension in~$D$ defined in the following way: \begin{inparaenum}[(i)] 
	\item $K := k$ if $D=k$; 
	\item $K := D$ if $D/k$ is an unramified extension; 
	\item $K := k$ if $D/k$ is a ramified extension and 
	\item if $D$ the quaternion algebra over $k$, then $K$ is the unramified quadratic extension of $k$ in $D$ normalized by $\varpi_D$.\end{inparaenum} 

Let $\nu_D = \val(\varpi_D)$. Then $\vD = \half$ if and only if $D/k$ is ramified or $D$ is the
 quaternion algebra over $k$.  Under this setting, $D=K$ if $\vD=1$ and
 $D = K \oplus \varpi_D K$ if $\vD = \half$. In all cases, $\fff_K = \fff_D$.

\subsection{} \label{sec:apartment}
We recall the explicit description of an apartment in $\cB(\bfG,k)$
(c.f. \cite[\SSS{2.9}]{BT4} and \cite[\SSS{2}-4]{BS}). 
Let $\ceil{r}$ denote the largest integer not greater than $r \in \bR$.
Let $n$ be the dimension of a maximally isotropic subspace in $V$. Let
$I := I^+\sqcup I^- \sqcup I^0$ where $I^+ = \set{1, \cdots, n}$,
$I^- = -I^+$ and $I^0$ is any index set with $\dim_DV - 2n$ elements.
Fix a basis $\set{e_i | i \in I}$ of~$V$ such that 
\begin{enumerate}[(a)]
\item $\inn{e_i}{e_j}_V = \inn{e_{-i}}{e_{-j}}_V = 0$ and $\inn{e_i}{e_{-j}}_V = \delta_{i,j}$ for all $i, j \in I^+$;
\item $e_i$ is anisotropic for $i\in I^0$  and $\inn{e_i}{e_j}_V = 0$ for $i\in I^0$
  and $i\neq j\in I$.
\end{enumerate}
For $i\in I^0$, we can choose $e_i$ such that
\begin{enumerate}[(i)]
\item  $\inn{e_i}{e_i}_V$ has valuation
either $0$ or $\vD$;
\item $\inn{e_i}{e_i}_V$ takes value either in $\foo_K$ or in $\varpi_D\foo_K$.  \footnote{This
    condition is non-trivial if $D/k$ is a ramified extension or $D/k$ is a quaternion algebra.}  
\end{enumerate}

Let $\bfS$ be the maximal $k$-split torus in $\bfG$ which stabilizes $e_iD$ for all $i\in I^+\sqcup I^-$ and fixes $e_j$ for all $j \in I^0$.
Then the apartment $\cA(\bfS,k)$ in $\cB(\bfG,k)$ corresponds to the set of self-dual lattice
functions which split under this basis. 
More precisely, if $\sL$ is in the apartment, then there is a (unique) tuple of real
numbers $(a_1,\cdots, a_n)\in \bR^n$ such that
\begin{equation} \label{eqsLr}
\sL_r = \bigoplus_{i\in I} e_i \fpp_D^{\ceil{(r-a_i)/\vD}}
\end{equation}  
where
\begin{inparaenum}[(i)]
\item $a_{-i} = - a_i$ for $i\in I^+$ and
\item $a_{i} = \half\val(\inn{e_i}{e_i}_V)$. 
\end{inparaenum}
In fact, $(a_1,\cdots,a_n) \mapsto \sL$ gives an identification of $\bR^n$ with
the apartment.

\medskip

\noindent \remark\ 
If $\bfG$ splits over an unramified extension of $k$, then the lattice function in \eqref{eqsLr} corresponds to a hyperspecial point in $\cA(\bfS,k)$ if and only if  $a_i/\nu_D \in b + \bZ$ for all $i \in I$ where $b = 0$ or $\half$.

\subsection{} We let $\sL$ be a lattice function as in \eqref{eqsLr} above.
We consider two cases. 
\trivial[h]{In the following discussion, $r$ or $\mu$ is not assumed to be in
  $\Jump(\sL)$ or $\Jump(\sB)$. If they are not jump points, the corresponding
  objects are just trivial one.}
\begin{enumerate}[1.]
\item[\underline{Case 1}.] First we assume that $\vD= 1$. In this case $D=K$. 
Let $[r]$ denote the coset $r+\bZ \in \bQ/\bZ$ and define
\[
\Vr := \sum_{a_i\equiv
  r\pmod{\bZ}} e_i K
\quad \text{and}\quad 
\sVr := \Vr\cap \sL_r = \sum_{a_i\equiv
  r\pmod{\bZ}}e_i\fpp_K^{r-a_i}.
\]
We make the following observations.
\begin{enumerate}[(a)]
\item 
The restriction of the Hermitian sesquilinear form to $\Vr$ is non-degenerate if $r\equiv 0$ or $\half\pmod{\bZ}$ and totally
isotropic if otherwise.
 
\item For $r\in \bR$, $\Vr$ is in perfect pairing with $\Vnr$. 
In particular, $\sV^0$ and $\sV^\half$ have $\epsilon$-Hermitian sesquilinear forms
  $\inn{}{}_V$ and $\inn{}{}_V \varpi_D^{-1}$ which are defined over $\foo_K$. 

\delete{
\item Clearly $\dim_K \Vr = \dim_{\fff_D}\sL_r/\sL_{r^+}$ and the natural inclusion $\sVr \hookrightarrow \sL_r$ induces an isomorphism 
$\sVr/\sV^{r}\fpp_D \iso \sL_r/\sL_{r^+}$. 
\item 
We have $V = \bigoplus_{r \in \bQ/\bZ}  \Vr$. We set
\[
\begin{split}
  \sM & := \rU(\sV^0)\times \rU(\sV^\half)\times \prod_{r\in (0,\half)}
  \GL_{\foo_K}(\sV^{r}) \\
  & \cong G_\sL \cap \left(\prod_{[r] \in \bQ/\bZ}\GL_{K}(\Vr)\right).
\end{split}
\]
Here $\sM$ is understood as a $\foo_k$-group scheme.
In particular, each factor is a classical group scheme over $\foo_k$ 
\item Let $M$ and $\sfM$ denote the generic fiber and special fiber of $\sM$
  respectively.  The lattices $\sVr$ give a vertex $y$ in the building of $M$.
  The Moy-Prasad filtration on $M$ given by this vertex is the same as
  $s\mapsto M\cap G_{\sL,s}$.  In particular, $\sM = M_{y}$,
  $M_{y,0^+} = M_{y,1}$, $\sfM = M_{y}/M_{y,0^+}$ and
  $\sfM = M_y/M_{y,0^+} \iso G_{\sL}/G_{\sL,0^+} = \sfG_\sL$.
}
\end{enumerate}

\smallskip

\item[\underline{Case 2}.]
Now assume $\vD = \half$. 
We define the $K$-module
\[
\Vr := \sum_{a_i\equiv r} e_i K + \sum_{a_i+\half\equiv r} e_i \varpi_D K
\]
and $\foo_K$-module
\[
\sVr := \Vr\cap \sL_r = \sum_{ a_i\equiv
  r}e_i\fpp_K^{r-a_i} + \sum_{ a_i +\half\equiv
  r } e_i\varpi_D \fpp_K^{r-a_i-\half}.
\]
We make the following observations.
\begin{enumerate}[(a)]
\item The two $K$-subspaces $\Vr$ and $\Vrhalf$ in $V$ are different. However $\Vr = \Vrhalf \varpi_D$ and $\sV^{r+\half} = \sV^r\varpi_D$.

\item The restriction of the Hermitian sesquilinear  form to $\Vr$ is non-degenerate if $r \equiv 0$ or $\quarter \pmod{\half \bZ}$ and totally isotropic if otherwise. 

\item For $r\in \bR$, $\Vr$ is in perfect pairing with $\Vnr$. In
  particular, there is an $\epsilon$-Hermitian sesquilinear  form
  $\inn{}{}_V$ and a $(-\epsilon)$-Hermitian sesquilinear  form $\inn{}{}_V \varpi_D^{-1}$ defined on $\sV^0$ and $\sV^{\quarter}$ respectively. Both forms are defined over~$\foo_K$. 
\end{enumerate}
\end{enumerate}

\medskip

\def\xyiso{\ar[r]^<>(.5){\sim}}
Thanks to the definitions of $\Vr$ and $\sVr$, the following holds for both Cases 1 and~2:
\begin{enumerate}[(i)]
\item $\dim_K \Vr = \dim_{\fff_D}\sL_r/\sL_{r^+}$ and the natural inclusion $\sVr \hookrightarrow \sL_r$ induces an isomorphism 
\begin{equation}\label{eq:sVr.iso}
\xymatrix{
\sVr/\sV^{r}\fpp_K \xyiso &\sL_r/\sL_{r^+}.}
\end{equation}

\item \label{it:V.dec} $V = \bigoplus_{[r]\in \bQ/\bZ} \Vr$.

\item Define an $\foo_k$-group scheme: 
\[
\sQ := \rU(\sV^0)\times \rU(\sV^{\half\vD})\times \prod_{r\in (0,\half\vD)} \GL_{\foo_K}(\sV^{r}).
\]
Let $Q$ and $\sfQ$ denote the generic fiber and special fiber of $\sQ$
  respectively. 
\item The lattices $\sVr$ give a vertex $y$ in the building of $Q$.
  Clearly, 
  \[
	\sQ = Q_{y},\quad Q_{y,0^+} = Q_{y,1} \quad \text{and}\quad \sfQ = Q_{y}/Q_{y,0^+}.
  \]
\item The natural action of $\sQ$ on $V$ identifies $Q$ with $G \cap \left(\prod_{[r]\in \bQ/\bZ}\GL_K(\Vr)\right)$
so that $\sQ = Q \cap G_{\sL}$.
\item The natural embedding $\sQ \rightarrow G_{\sL}$ induces an isomorphism of
  $\fff$-groups
  \[
  \xymatrix{
    \sfQ = Q_{y}/Q_{y,1} \xyiso & G_{\sL}/G_{\sL,0^+} = \sfG_\sL
  }
  \]
  which is compatible with \eqref{eq:sVr.iso}. This follows for the fact that
  the both sides are isomorphic to the right hand side of \eqref{eq:sfG}.
\savemyenumi
\end{enumerate}
For $V'$, we likewise divide into two cases and define similar notations
$V'^{[r]}$, $\sV'^r$, $Q'$, $\sQ'$ etc as above.

\subsection{} 

We recall $\sB = \sL \otimes \sL'$.  For $\mu\in \bR$, we define
$X^{[\mu]} = \sum_{[t]+[t'] = [\mu]} V^{[t]}\otimes_K V'^{[t']}$. 
Then
\begin{enumerate}[(i)]
\resumemyenumi
\item $W = \bigoplus_{[\mu]\in \bQ/\bZ} X^{[\mu]}$ 
\item $\sX_\mu := \sum_{t+t' = \mu} \sV^t\otimes_{\foo_K} \sV'^{t'}$ equals
$X^{[\mu]}\cap \sB_{\mu}$.
\savemyenumi
\end{enumerate}
Using the natural inclusion $\sX_\mu \hookrightarrow \sB_{\mu}$, we have
\begin{enumerate}[(i)]
\resumemyenumi
\item $\sX_{\mu+1} = \sX_{\mu^+} := X_{[\mu]}\cap \sB_{\mu^+}$.

\item \label{it:sX.iso} $\sX_{\mu}/\sX_{\mu+1} \iso  \sX_{\mu}/\sX_{\mu^+} \iso \sB_{\mu}/\sB_{\mu^+}$.

\end{enumerate}

\smallskip

\begin{proof}[Proof of \Cref{lem:split}]
We recall \eqref{equ9} where 
\[
\sfY := \sB_{\fracdmm:\fracdmm^+}, \quad  \sfW := \sB_{-\fracdmm:\fracdmm^+} \quad \text{and} \quad \sfX := \sB_{-\fracdmm:-\fracdmm^+}.
\] 
Let $\sfX' := \sX_{-\fracdmm}/\sX_{-\fracdmm}\fpp_K$. 
Clearly $\sfX' \iso \sfX$ by \eqref{it:sX.iso}. 

Note that $m\geq 2$. So $\fracmm<1$.  The inclusion
$\sX_{-\fracdmm}\hookrightarrow \sB_{-\fracdmm}$ gives an embedding
\begin{equation} \label{eq:split.X}
\xymatrix{
 \sfX' = \sX_{-\fracdmm}/\sX_{-\fracdmm+1} \ar@{^(->}[r]& \sB_{-\fracdmm}/\sB_{\fracdmm^+} = \sfW
} 
\end{equation}
which splits the quotient map $\sfW \twoheadrightarrow \sfX$. 
 The embedding $\sQ \times \sQ' \hookrightarrow G_\sL\times G'_{\sL'}$ induces a splitting
  of $\sfP \twoheadrightarrow \sfJ$:
\begin{equation}\label{eq:split.J}
\xymatrix{
\sfJ \riso \sfQ\times \sfQ' =  \sQ/\sQ_{1} \times
\sQ'/\sQ_{1}' \ar@{^(->}[r] & G_{\sL}/G_{\sL,\fracmm^+} \times
G'_{\sL'}/ G'_{\sL',\fracmm^+} = \sfP.
}
\end{equation}
Note that
$\sfY$, $\sfX'$ and $\sfW$ are natural $\sfQ \times \sfQ'$-modules and \eqref{eq:split.X}
is an $\sfQ \times \sfQ'$-equivariant embedding.  We get a decomposition
$\sfW = \sfX' \oplus \sfY$ as $\sfJ$-modules under the
splitting \eqref{eq:split.J}. 
\end{proof}

\section{Matrix calculations} \label{app:M}

 In this appendix, we prove \Cref{LJ1} and \Cref{lem:Xll}.

\subsection{Proof of \Cref{LJ1}} \label{app:M0} 
We construct below an $\AAA$ defined over
an algebraically closed field $\barkk$ which satisfies the lemma. The
lemma and the proof below is valid for any field provided
$(G,G')$ is an irreducible dual pair such that $G$ and $G'$ are both
split.

There are only several cases.
\begin{asparaenum}[1.]
\item $(G,G') = (\GL(n), \GL(n'))$ with $n \leq n'$.  We
can identify \begin{inparaenum}[(a)]
\item $W = \Mat_{n,n'} \oplus \Mat_{n,n'}$,
\item $(x,y)^\star  = (y,-x)$ for $(x,y) \in W$, 
\item $M(x,y) = xy^\top \in \fgl(n)$ and
\item  $M(x,y) = y^\top x \in \fgl(n')$.
\end{inparaenum}
We set 
\[
\AAA = \Set{ w= (\begin{pmatrix} a & 0
\end{pmatrix}, \begin{pmatrix} b & 0
\end{pmatrix}) |\begin{array}{l}a = \diag(a_1,\cdots, a_n) \\
                  b = \diag(b_1,\cdots, b_n)
\end{array}
\text{ with } a_i, b_i  \in \barkk}.
\] 
For any $w \in \AAA$, $M(w) = ab$ and
$M'(w) = \begin{pmatrix} ab & 0 \\ 0 & 0 
\end{pmatrix}$.

\item $(G,G')= (\Sp(2n), \rO(2n'+1))$ with $n \leq n'$.  We can choose
  suitable bases so that $V = \barkk^{2n}$ and $V' = \barkk^{2n'+1}$
  such that $\inn{v_1}{v_2}_V = v_1^{\top} J v_2$ and
  $\inn{v'_1}{v'_2}_{V'} = v'^\top_1J'v'_2$ where
\[
J = \begin{pmatrix} 0& I_n \\ -I_n & 0
\end{pmatrix} \quad \text{and} \quad
J' = \begin{pmatrix}0 & 0 & I_n  \\ 0 & I_{2n'-2n+1} & 0\\ I_n &0& 0
\end{pmatrix}.
\]
Now we can identify \begin{inparaenum}[(a)]

\item $W = M_{2n'+1,2n}$,

\item $w^\star  = J^{-1} w^{\top} J'$, 

\item $M(w) = w^\star w$ and

\item $M'(w) = w w^\star$.
\end{inparaenum}
We consider 
\[
\AAA = \Set{\begin{pmatrix}a& 0 \\ 0& 0\\0 & -b
\end{pmatrix}|\begin{array}{l}a = \diag(a_1,\cdots, a_n), \\
                b = \diag(b_1,\cdots, b_n),
              \end{array},\text{ with } a_i, b_i  \in \barkk }.
\] 
For any $w\in \AAA$,  $M(w) =  \begin{pmatrix}
ab & 0\\ 0& -ab
\end{pmatrix}$
and $M'(w) = \begin{pmatrix} ab & 0& 0 \\ 0& 0 & 0\\ 0 & 0 &-ab
\end{pmatrix}$.

\item We leave the all other cases where
  $(G,G') = (\Sp(2n),\rO(2n'+2))$, $(\rO(2n),\Sp(2n'))$ or
  $(\rO(2n+1),\Sp(2n'))$ where $n\leq n'$ to the reader.  The formulas
  are similar to 2. \qed
\end{asparaenum}

\subsection{Proof of \Cref{lem:Xll}}\label{app:Xll} 
  We translate everything to the left hand side of \eqref{eq:tMM} and denote the
  images of $\bar{w}$, $\lambda$, $\lambda'$, $\sfXll$, $\sfSl$, $\cdots$, by
  $\tww$, $\tlambda$, $\tlambda'$, $\tsfXll$, $\tsfSl$, $\cdots$
  respectively. We also transport implicitly the Galois actions. Then $\tlambda$
  and $\tlambda'$ are regular semisimple elements.  It is enough to prove the
  statements for $\tlambda$ and $\tlambdap$.

  \begin{asparaenum}[(i)]
  \item First we assume that $\tlambda \in \Hom_\fffDE(\sfV,\sfV)$ is full
    rank. In this case $\tww \in \Hom_\fffDE(\sfV,\sfV')$ is full rank too.  By
    Witt's theorem, $\tsfM'^{-1}(\tlambda')$ is a single free $\sfGE$-orbit.
    Let $\tww'\in \tsfXll$. Then there is a unique $g \in \sfGE$ such that
    $\tww' = g \cdot \tww$. Clearly $g\in \Stab_{\sfGE}(\tlambda)$. For every
    $\sigma \in \Gal(E/k)$,
\begin{equation}\label{eq:btwp}
  g \cdot \tww = \tww' = \sigma(\tww') = \sigma(g)\cdot \sigma(\tww) =
  \sigma(g)\cdot \tww.
\end{equation}
Since $\sfGE$ acts freely, we have $g = \sigma(g)$. Hence
$g\in \tsfSl = (\Stab_\sfGE(\tlambda))^{\Gal(E/k)}$.  This proves (i) in these
cases.

Now we suppose that $\tlambda$ is not full rank i.e. Case (E). This
only occurs for unitary groups of equal rank. We refer to the \Cref{app:M0}
for the notation.  In this case,
$\tsfW = M_{nn}(\bfff) \oplus M_{nn}(\bfff)$ are two copies of $n$ by
$n$ matrices and $\tlambda$ is of rank $n-1$.  There is an element in
$\Gal(E/k)$ exchanging the two components of $\tww = (A,B)$, hence $A$
and $B$ have the same rank $n-1$. The group $\sfGE$ is a general
linear group. Let $\sfSLE$ be the special linear group in $\sfGE$.
Let $\tww'\in \tsfXll$. Let
$\sfSLE_{\tlambda} := \Stab_{\sfSLE}(\tlambda)$. It is straightforward
to check that $\tww'$ and $\tww$ are in the same
$\sfSLE_{\tlambda}$-orbit on which $\sfSLE_{\tlambda}$ acts freely.
Let $g\in \sfSLE_{\tlambda}$ such that $\tww' = g\cdot \tww$.  Again
by~\eqref{eq:btwp}, $g$ is Galois invariant, i.e. $g\in \tsfSl$.  This
proves (i).

\item Let $w \in \sB_{\sz}$ be a lift of $\bw$. Without loss of
  generality, we may assume that $w$ is of full rank. Then
  $\Gamma = M(w)$ and $\Gamma' = M'(w)$ are lifts of $\lambda$ and
  $\lambda'$ respectively. 
	Let $\bfSl$ (resp. $\bfSlp$) be the
  stabilizer of $\Gamma$ (resp. $\Gamma'$) in $G$ (resp. $G'$). We
  recall that $\bfSl$ is anisotropic so $\cB(\bfSl) = \set{ x }$ as
  shown in the proof of \Cref{prop:MKT}.  
	This implies that
  $\bfSl \subseteq G_\sL$ and $\bfSlp\subseteq G'_{\sL'}$.  Using the
  same argument and Witt's theorem as in (i), for every
  $g' \in \bfSlp$ there is a unique $g \in \bfSl$ such that
  $g' w = g^{-1} w$. The map $\talpha \colon \bfSlp \rightarrow \bfSl$
  given by $g' \mapsto g$ is a surjective homomorphism.  Note that
  $\bfSl$ (resp. $\bfSlp$) surjects onto $\sfSl$ (resp. $\sfSlp$)
  since $\Gamma$ (resp. $\Gamma'$) is a good element
  (c.f. \cite[Corollary 2.3.5 and Lemma~2.3.6]{KM}). Then $\talpha$
induces a homomorphism $\alpha \colon \sfSlp \rightarrow \sfSl$.

\item This follows from the proofs of (i) and (ii).
\item Note that $\Im\tww^\star = \Im \tlambda \subsetneq \sfV$. Let $g \in \tsfSl$. Then
  $g\in \tbSw$ if and only if $g|_{\Im \tww^*} =\id$ if and only if $g|_{\Im\tlambda} = \id$ if and only if $g \in
  \tbSl$, i.e. $\tbSw = \tbSl$. Since $\tsfXll \cong \tsfSl/\tbSw$, The last assertion is clear. \qed
\end{asparaenum}

\section{Lattice model and splitting}\label{app:SPLIT}

\subsection{}\label{app:Lsplit} 
\def\da{{\mathrm{d}a}} Let $\Sp(W)$ be a symplectic group of a
symplectic space $W$. \Cref{prop:liftSpA} follows from \Cref{lem:SS}
below.  One may compare the lemma with \cite[\SSS{4.1}]{SZ} and
\cite{Pan01}.
\begin{lemma} \label{lem:SS} Let $A_1$ and $A_2$ be two self-dual
  lattices in $W$.  For $i = 1,2$, let
  $\omega_{A_i}\colon \Sp(W) \rightarrow \Mp(W)$ be the section
  defined by the lattice model $\sS(A_i)$ as in \eqref{eq:LA}. Then
\[
\omega_{A_1}(g) = \omega_{A_2}(g) \quad \forall g\in \Sp_{A_1}\cap \Sp_{A_2}. 
\]
\end{lemma}
\begin{proof}
  We have an intertwining operator $\Xi \colon \sS(A_1) \rightarrow \sS(A_2)$
  given by $(\Xi\,f)(w) = \int_{A_2}\psi(\half\inn{a}{w}) f(w+a) \da$ between
  the two lattice models. This intertwining operator is unique up to scalar.

Let $g\in \Sp_{A_1}\cap \Sp_{A_2}$. 
Since $g\colon A_2 \rightarrow A_2$ is measure preserving,
\[
\begin{split}
  ((\omega_{A_2}(g)\circ \Xi f)(w) 
  =& \int_{A_2}\psi(\half \inn{a}{g^{-1}w}) f(g^{-1}w +a) \da\\
  = & \int_{A_2}\psi(\half \inn{ga}{w}) f(g^{-1}w +a) \da \\
  = & \int_{A_2}\psi(\half \inn{a}{w}) f(g^{-1}w + g^{-1}a) \da \\
  = & (\Xi \circ \omega_{A_1}(g) f)(w).
\end{split}
\]
This proves the lemma.
\end{proof}

\subsection{} \label{sec:wplus}
Let $x\in \cB(\bfG,k)$. We pick any $x'\in \cB(\bfG',k)$ and let $\sL$ and
$\sL'$ be the lattice functions corresponding to $x$ and $x'$ respectively.  Let
$\sB = \sL\otimes \sL'$ be the tensor product lattice function and 
$A$ be any
self-dual lattice such that $\sB_{0^+} \subseteq A \subseteq \sB_{0}$.
We have $G_{x,0^+}$ stabilizes $A$, i.e.  $G_{x,0^+} \subseteq \SpA$ (see also \cite[\S
3.3.2]{Pan}).  As a corollary of \Cref{lem:SS}, the
lattice models give a canonical splitting
\[
\omega_+ \colon  \bigcup_{x\in \cB(\bfG,k)} G_{x,0^+} \longrightarrow \tG.
\]

\end{document}